\documentclass[11pt]{amsart}

\usepackage{amsmath, amsthm, amssymb, amsfonts, enumerate}
\usepackage{dsfont}
\usepackage{color}
\usepackage{geometry}
\usepackage{todonotes}
\usepackage{subfigure}

\newtheorem{theorem}{Theorem}[section]
\newtheorem{remark}[theorem]{Remark}

\newtheorem{lemma}[theorem]{Lemma}
\newtheorem{proposition}[theorem]{Proposition}
\newtheorem{corollary}[theorem]{Corollary}

\DeclareMathOperator*{\I}{Int}

\newcommand{\LL}{ \mathcal{L}}

\title[Optimal Installation of Solar Panels with Price Impact]{Optimal Installation of Solar Panels with Price Impact: a Solvable Singular Stochastic Control Problem}
\author[Koch, Vargiolu]{Torben Koch*, Tiziano Vargiolu}\thanks{*Corresponding author.}

\address{T.~Koch: Center for Mathematical Economics (IMW), Bielefeld University, Universit\"atsstrasse 25, 33615, Bielefeld, Germany}
\email{\href{mailto:t.koch@uni-bielefeld.de}{t.koch@uni-bielefeld.de}}
\address{T.~Vargiolu: Department of Mathematics, University of Padova, Via Trieste 63, 35131 Padova, Italy}

\email{\href{mailto:vargiolu@math.unipd.it}{vargiolu@math.unipd.it}}

\date{\today}

\numberwithin{equation}{section}

\begin{document}

\begin{abstract} 
We consider a price-maker company which generates electricity and sells it in the spot market. The company can increase its level of installed power by irreversible installations of solar panels. In absence of the company's economic activities, the spot electricity price evolves as an {Ornstein-Uhlenbeck process}, and therefore it has a mean-reverting behavior. The current level of the company's installed power has a permanent impact on the electricity price and affects its mean-reversion level. The company aims at maximizing the total expected profits from selling electricity in the market, net of the total expected proportional costs of installation. This problem is modeled as a \emph{two-dimensional degenerate singular stochastic control problem} in which the installation strategy is identified as the company's control variable. We follow a \emph{guess-and-verify approach} to solve the problem. We find that the optimal installation strategy is triggered by a curve which separates the \emph{waiting region}, where it is not optimal to install additional panels, and the \emph{installation region,} where it is. Such a curve depends on the current level of the company's installed power, and is the unique strictly increasing function which solves a first-order {ordinary differential equation} (ODE). Finally, our study is complemented by a numerical analysis of the dependency of the optimal installation strategy on the model's parameters.
\end{abstract}

\maketitle

\smallskip


\smallskip

{\textbf{MSC2010 subject classification}}: 93E20; 49L20; 91B70; 60G99.

\smallskip





\section{Introduction}
\label{introduction}

This paper proposes a model in which a company can increase its current electricity production by irreversible investments in solar panels, while maximizing net profits. Irreversible investment problems have been widely studied in the context of real options and optimal capacity expansion. Related models in the economics literature are, for example, \cite{Bertola} and the monography \cite{Dixit}. Other relevant papers appearing in the mathematical literature are \cite{aidetal2015,Ferrari,Chiarolla,Federico,Ferrari3,Ferrari2,Lokka,Oksendal,Riedel,Steg}, among many others.

We consider an infinitely-lived profit maximizing company which is a large player in the market. {The company can install solar panels in order to increase its production level of electricity up to a given maximum level. The electricity generated will immediately be sold in the market, and while installing additional panels, the company incurs constant proportional costs.} 
As it is assumed that the company is a large market player, its activities have an impact on the electricity price. In particular, we assume that the long-term electricity price level is negatively affected by the current level of installed power; that is, the electricity price will tend to move towards a lower price level if the electricity production is increased. Therefore, the company has to install solar panels carefully in order to avoid permanently low electricity prices which clearly decrease the marginal profits from selling electricity in the market.

The mathematical formulation of the model leads to a \textit{two-dimensional degenerate singular stochastic control problem} (see, for example, \cite{K85,karatzas2,Karatzasetal00} as early contributions) whose components are the electricity price and the current level of installed power which is purely controlled. To the best of our knowledge, this paper is the first which provides the complete explicit solution to a two-dimensional degenerate singular stochastic control problem in which the drift of one component of the state process (the electricity price) is linearly affected by the monotone process giving the cumulative amount of control (the level of installed power). In our model the electricity price evolves as an Ornstein-Uhlenbeck process, and dealing with such a process makes the problem more difficult in comparison to, for example, a geometric Brownian motion setting, due to the unhandy and non-explicit expressions of the fundamental solutions of the second-order ordinary differential equation involving the infinitesimal generator of the underlying Ornstein-Uhlenbeck process. It is worth noticing that our mathematical formulation shares similarities with the recent article \cite{Federico2019} in which a central bank can choose a control of bounded variation for managing the inflation. The methodology and results of \cite{Federico2019} are indeed different with respect to ours:  in fact, in that paper the authors provide a theoretical study of the structure and regularity of the value function using viscosity theory and free-boundary analysis, but do not construct an explicit solution, as instead we do.

Price impact models have gained the interest of many researchers in recent years. Some of these works are also formulated as a \textit{singular stochastic control} problem and study questions of optimal execution: \cite{Becherer2} and \cite{Becherer} take into account a multiplicative and transient price impact, whereas \cite{guo} considers an exponential parametrization in  a geometric Brownian motion setting allowing for a permanent price impact. Also, a price impact model working with \textit{singular stochastic controls} has been studied by \cite{Zervos}, motivated by an irreversible capital accumulation problem with permanent price impact, and by \cite{Koch2}, in which the authors consider an extraction problem with Ornstein-Uhlenbeck dynamics and transient price impact. 
In all of the aforementioned papers on price impact models dealing with \textit{singular stochastic controls} \cite{Zervos,Becherer2,Becherer,Koch2,guo}, the agents' actions can lead to an immediate jump in the underlying price process, whereas in our setting, it cannot. 
Finally, \cite{CFSV,CarJai} show how to incorporate a market impact due to cross-border trading in electricity markets, and \cite{RVG} models the price impact of wind electricity production on power prices. 

In our model the firm's installation strategy is represented by an increasing control, possibly non-absolutely continuous, and we take into account a running payoff function which depends linearly on the level of installed power and on the electricity price. Following an educated guess for a classical solution to the associated Hamilton-Jacobi-Bellman (HJB) equation, and imposing $C^{2,1}-$regularity of the value function, we show that the optimal installation rule is triggered by a threshold which is a function of the current level of installed power, and we provide a closed-form expression of the value function. The threshold, also called free boundary, uniquely solves an{ ordinary differential equation} (ODE) for which we implement a numerical solution. Then, we characterize the geometry of the \emph{waiting} and \emph{installation} regions. We show that the optimal installation strategy is such that the company keeps the state process inside the \emph{waiting} region. In particular, the state process is pushed towards the free boundary by installing a block of solar panels immediately, if the initial electricity price is above the critical threshold (if the maximum level of installed power, that the company is able to reach, is not sufficiently high, the company will immediately install the maximum number of panels). Thereafter, the joint process will be reflected along the free boundary. The construction of the reflected diffusion relies on ideas in \cite{Chiarolla2} that are based on the transformation of probability measures in the spirit of Girsanov. The uniqueness of the optimal diffusion process then follows by the global Lipschitz continuity of our free boundary. Our results are finally complemented by a numerical discussion of the dependency on the model parameters. We find, for example, that a higher mean-reversion level of the fundamental price process leads to a quicker installation of solar panels.

From the modeling point of view, it is common in the literature to represent electricity prices via a mean-reverting behavior, and to include (jump) terms to incorporate seasonal fluctuations and daily spikes, cf. \cite{Borovkova,CarFig,German,Weron} among others. Here, we do not represent the spikes and seasonal fluctuations, with the following justification: the installation time of solar panels usually takes several days or weeks, which makes the company indifferent to daily or weekly spikes. Also, the high lifespan of solar panels and the underlying infinite time horizon setting allow us to neglect the seasonal patterns. We therefore assume that the fundamental electricity price has solely a mean-reverting behavior, and evolves according to an Ornstein-Uhlenbeck process\footnote{We allow for negative prices by modeling the electricity price via an Ornstein-Uhlenbeck process. Indeed, negative electricity prices can be observed in some markets, for example in Germany, cf. \cite{NYTimes}.}. We are also neglecting the stochastic and seasonal effects of solar production. In fact, solar panels obviously do not produce power during the night, produce less in winter than in summer (these two effects could be covered via a deterministic seasonal component), and also produce less when it is cloudy (this should be modeled with a stochastic process). Since here we are interested in a long-term optimal behaviour, we interpret the average electricity produced in a generic unit of time as proportional to the installed power. All of this can be mathematically justified if we interpret our fundamental price to be, for example, a weekly average price as e.g. in \cite{BPPB,GPP}, who used this representation exactly to get rid of daily and weekly seasonalities. 

The rest of the paper is organized as follows. In Section \ref{sec:setting} we introduce the setting and formulate the problem. In Section \ref{sec:PreResVerTheorem} we provide preliminary results and a Verification Theorem. Then, in Section \ref{sec:Sol} we derive a characterization  of the free boundary via an ODE, and the explicit solution is constructed. Finally, Section \ref{sec:NumImp} provides a numerical implementation, and studies the dependency of the free boundary with respect to the model parameters.


\section{Model and Problem Formulation}
\label{sec:setting}
Let $(\Omega, \mathcal{F}, \mathbb{F}:=(\mathcal{F}_t)_{t\geq 0}, \mathbb{P})$ be a filtered probability space with a filtration $\mathbb{F}$ satisfying the usual conditions, and carrying a standard one-dimensional $\mathbb{F}$-Brownian motion $W$.

We consider an  infinitely-lived company which installs solar panels and sells the electricity produced by those panels instantaneously in the spot market. In absence of the company's economic activities, the fundamental electricity price $(X^x_t)_{t\geq 0}$ evolves stochastically according to an Ornstein-Uhlenbeck dynamics 
\begin{align}\label{Xuncontrolled}
dX^x_t = \kappa\big(\mu-X^x_t\big)dt + \sigma dW_t,\quad X^x_0 = x > 0,
\end{align} for some constants $\mu\in\mathbb{R}$ and $\kappa,\sigma> 0$.

The level of installed power can be increased at constant proportional cost $c\geq 0$ due to the installation costs of panels. It is assumed that the firm cannot reduce the number of solar panels, thus the installation is irreversible. The current level of installed power is described by the process $(Y^{y,I}_t)_{t\geq 0}$, which is given by 
\begin{align}
Y^{y,I}_t=y+I_t,
\end{align}
where the initial level of installed power is denoted by $y\geq 0$, and $I_t$ is identified as the company's control variable: it is an $\mathbb{F}$-adapted nonnegative and increasing c\`{a}dl\`{a}g process $I=(I_t)_{t\geq 0}$, where $I_t$ represents the total power installed within the interval $[0,t]$. In the following, $(I_t)_{t\geq 0}$ is also referred to as the installation strategy. Moreover, we assume that the level of installed power cannot exceed a given $\bar{y}\in[y,\infty)$ since, for example, only a finite number of solar panels can be installed. The set of admissible installation strategies is therefore defined as 
\begin{align*}
\mathcal{I}^{\bar{y}}(y):=\{I:\Omega\times[0,\infty)\mapsto[0,\infty)&\text{ : $(I_t)_{t\geq 0}$ is $\mathbb{F}$-adapted, }t\mapsto I_t \text{ is increasing, c\`{a}dl\`{a}g,}\\
&\text{ with }I_{0-}=0\leq I_t\leq\bar{y}-y\text{ a.s.}\}.
\end{align*} 
We write $\mathcal{I}^{\bar{y}}(y)$ in order to stress the dependency on both the initial level of installed power $y$ and the maximum possible level $\bar{y}$.

We assume that the current level of electricity production, which is proportional to $Y^{y,I}_t$, affects the electricity market price. In particular, when following an installation strategy $I\in\mathcal{I}^{\bar{y}}(y)$, the mean level of the market price $X$ is instantaneously reduced at time $t$ by ${\beta} Y^{y,I}_t$, for some $\beta>0$, and the spot price $X^{x,y,I}$ thus evolves as 
\begin{align}\label{affectedX}
dX^{x,y,I}_t = \kappa\bigg((\mu-\beta Y^{y,I}_t)-X^{x,y,I}_t\bigg)dt + \sigma dW_t,\quad X^{x,y,I}_{0-} = x > 0.
\end{align}
The company aims at maximizing the total expected profits from selling electricity in the market, net of the total expected costs of installation. That is, the company aims at determining
\begin{align}\label{ValueFnc} 
V(x,y):=\sup_{I\in\mathcal{I}^{\bar{y}}(y)}\mathcal{J}(x,y,I),\quad (x,y)\in\mathbb{R}\times[0,\bar{y}],
\end{align}
where for any $I\in\mathcal{I}^{\bar{y}}(y)$
\begin{align}\label{PC}
\mathcal{J}(x,y,I):=\mathbb{E}\bigg[\int_{0}^{\infty}e^{-\rho t}X_t^{x,y,I}\left(\alpha Y^{y,I}_t\right) dt-c\int_{0}^{\infty}e^{-\rho t}dI_t\bigg],\quad\alpha>0.
\end{align}
In \eqref{PC}, the parameter $\alpha$ is the proportional factor between the average electricity produced in a generic unit of time and the current level of installed power. Thus, the running gain $\alpha X_t^{x,y,I} Y_t^{y,I}$ can be viewed as a weekly-averaged revenue deriving from solar production, here represented in continuous time as the life span of a typical solar panel is of several years. 

For the sake of simplicity, we set $\alpha=1$ in the following. In fact, the problem of finding an optimal control $I\in\mathcal{I}^{\bar{y}}(y)$ in \eqref{PC} does not change for $\alpha>0$ upon introducing a new cost factor $\tilde{c}=\frac{c}{\alpha}.$

\section{A Verification Theorem}
\label{sec:PreResVerTheorem}
The aim of this section is to provide a verification theorem which characterizes the solution to our problem.

A non-installation strategy is denoted by the function $I^0\equiv0$, and we {indicate} the electricity price process {implied by} $I^0$ by $(X^{x,y}_t)_{t\geq 0}$, that is $X^{x,y}_t\equiv X^{x,y,I^0}_t$. Then, the expected profits of the firm following a non-installation strategy is described by the function $R:\mathbb{R}\times[0,\bar{y}]\mapsto\mathbb{R}$ such that
\begin{align}\label{PS}
R(x,y):=\mathcal{J}(x,y,I^0)=\mathbb{E}\bigg[\int_{0}^{\infty}e^{-\rho t}X_t^{x,y}y dt\bigg]=\frac{x{y}}{\rho+\kappa}+\frac{\mu\kappa{y}}{\rho(\rho+\kappa)}-\frac{\kappa\beta {y}^2}{\rho(\rho+\kappa)},
\end{align} 
The following preliminary result provides a growth condition and a monotonicity property of the value function $V$, and its connection to the function $R$. The proof of the proposition can be found in the appendix.
\begin{proposition}\label{GrowthV}
	There exist a constant $K>0$ such that for all $(x,y)\in\mathbb{R}\times[0,\bar{y}]$ one has
	\begin{align}\label{eq13}
	|V(x,y)|\leq K\big(1+|x|\big).
	\end{align}
	Moreover, $V(x,\bar{y})=R(x,\bar{y})$, and $V$ is increasing in $x$.
\end{proposition}

In a next step we derive the Hamilton-Jacobi-Bellman (HJB), a particular partial differential equation which characterizes the solution to our problem. 

For given and fixed $y\geq 0$, let $\mathcal{L}^y$ be the infinitesimal generator of the diffusion $X^{x,y}$ given by the second order differential operator
\begin{align}\label{InfinitesOp}
\LL^y u(x,y):=\frac{1}{2}\sigma^2\frac{\partial^2}{\partial x^2}u(x,y) + \kappa\Big((\mu-{\beta}y)-x\Big)\frac{\partial}{\partial x}u(x,y),
\end{align}
where $u(\cdot,y)\in C^{2}(\mathbb{R})$.

The HJB equation, for singular control problems as this one, follows this heuristic argument. At time zero, the firm has two possible options: either it waits for a short time period $\Delta t$, in which the firm does not install additional panels and gains running profits from selling $y$ units of electricity in the market, or it can install solar panels immediately in order to increase its level of installed power. After each of these actions the firm behaves optimally. Suppose that the firm follows the first action. Since this action is not necessarily optimal,
it is associated to the inequality
\begin{align}\label{eq9}
V(x,y)\geq\mathbb{E}\bigg[\int_{0}^{\Delta t}e^{-\rho s}{X_s^{x,y}}yds+e^{-\rho \Delta t}V(X^{x,y}_{\Delta t},y)\bigg],\quad(x,y)\in\mathbb{R}\times[0,\bar{y}).
\end{align}
Employing It\^o's formula to the last term of the right-hand side of \eqref{eq9}, dividing by $\Delta t$, and then letting $\Delta t\rightarrow 0$, we obtain
$$\mathcal{L}^yV(x,y)-\rho V(x,y)+xy\leq 0,\quad(x,y)\in\mathbb{R}\times[0,\bar{y}).$$
Now, suppose the firm follows the second option, i.e. to increase its level of installed power by $\varepsilon>0$ units and then to continue optimally. This action is associated to 
$$V(x,y)\geq V(x,y+\varepsilon)-c\varepsilon,$$
which in turn, by dividing by $\varepsilon$ and letting  $\varepsilon\downarrow 0$, implies 
$$V_y(x,y)-c\leq 0.$$

The previous observations suggest that $V$ should identify with an appropriate solution $w$ to the HJB equation
\begin{align}\label{HJB1}
\max\Big\{\mathcal{L}^yw(x,y)-\rho w(x,y)+xy,w_y(x,y)-c\Big\}=0,\quad(x,y)\in\mathbb{R}\times[0,\bar{y}),
\end{align} with boundary condition $$w(x,\bar{y})=R(x,\bar{y}).$$

With reference to \eqref{HJB1}, we introduce the \emph{waiting region} 
\begin{align}\label{WROld}
\mathbb{W}:=\{(x,y)\in\mathbb{R}\times[0,\bar{y}):\mathcal{L}^yw(x,y)-\rho w(x,y)+xy=0,\,w_y(x,y)-c<0\},
\end{align}
where we expect not to be optimal to install additional solar panels, and the \emph{installation region} 
\begin{align}\label{IROld}
\mathbb{I}:=\{(x,y)\in\mathbb{R}\times[0,\bar{y}):\mathcal{L}^yw(x,y)-\rho w(x,y)+xy\leq0,\,w_y(x,y)-c=0\},
\end{align} 
where we expect it to be.

We move on by proving a Verification Theorem. It shows that an appropriate solution to the HJB equation \eqref{HJB1} identifies with the value function, if an admissible installation strategy exists which keeps the state process $(X,Y)$ inside the waiting region $\overline{\mathbb{W}}$ with minimal effort, i.e. by increasing the level of installed power whenever $(X,Y)$ enters the installation region $\mathbb{I}$. Here, we have denoted by $\overline{\mathbb{W}}$ the closure of $\mathbb{W}$.

\begin{theorem}[Verification Theorem]\label{VerificationTheorem}
	Suppose there exists a function $w:\mathbb{R}\times[0,\bar{y}]\mapsto\mathbb{R}$ such that $w\in C^{2,1}(\mathbb{R}\times[0,\bar{y}])$ solves the HJB equation \eqref{HJB1} with boundary condition $w(x,\bar{y})=R(x,\bar{y})$, and satisfies the growth condition 
	\begin{align}\label{eq15}
	|w(x,y)|\leq K\big(1+|x|\big),
	\end{align} for a constant $K>0$. Then $w\geq v$ on $\mathbb{R}\times[0,\bar{y}]$.\\ 
	Moreover, suppose that for all initial values $(x,y)\in\mathbb{R}\times[0,\bar{y})$, there exists a process $I^\star\in\mathcal{I}^{\bar{y}}(y)$ such that
	\begin{align}\label{eq2VT}
	(X_t^{x,y,I^\star},Y^{y,I^\star}_t)\in{\overline{\mathbb{W}}},\quad\text{for all }t\geq 0,\, \text{$\mathbb{P}$-a.s.},\\\label{eq3VT}
	I_t^\star=\int_{0-}^t\mathds{1}_{\{(X_s^{x,y,I^\star},Y^{y,I^\star}_s)\in\mathbb{I}\}}dI^\star_s,\quad\text{for all }t\geq 0,\, \text{$\mathbb{P}$-a.s.}
	\end{align}
	Then we have
	$$V(x,y)=w(x,y),\quad (x,y)\in\mathbb{R}\times[0,\bar{y}],$$
	and $I^\star$ is optimal; that is, $V(x,y)=\mathcal{J}(x,y,I^\star)$.
\end{theorem}

\begin{proof} {Since we have $w(x,\bar{y})=R(x,\bar{y})=V(x,\bar{y})$ by assumption, we let $y<\bar{y}$.} In a first step, we prove that $w\geq v$ on $\mathbb{R}\times[0,\bar{y})$, and then in a second step, we show that $w\leq v$ on $\mathbb{R}\times[0,\bar{y})$ and the optimality of $I^\star$ satisfying \eqref{eq2VT} and \eqref{eq3VT}.\vspace{0.25cm}
	
\emph{Step 1.} Let $(x,y)\in\mathbb{R}\times[0,\bar{y})$ be given and fixed, and $I\in\mathcal{I}^{\bar{y}}(y)$. For $N>0$ we set $\tau_{R,N}:=\tau_{R}\wedge N,$ where $\tau_{R}:=\inf\{s>0:X^{x,y,I}_s\notin (-R,R)\}$. In the following, we write $\Delta I_s:=I_s-I_{s-}$, $s\geq 0$, and $I^c$ denotes the continuous part of $I\in\mathcal{I}^{\bar{y}}(y)$. By an application of It\^o's formula, we have
	\begin{align}\label{eq14}
	\begin{split}
	&e^{-\rho \tau_{R,N}}w(X^{x,y,I}_{\tau_{R,N}},Y^{y,I}_{\tau_{R,N}})-w(x,y)\\
	=&\int_{0}^{\tau_{R,N}}e^{-\rho s}\Big(\mathcal{L}^yw(X^{x,y,I}_s,Y^{y,I}_s)-\rho w(X^{x,y,I}_s,Y^{y,I}_s)\Big)ds+\underbrace{\sigma\int_{0}^{\tau_{R,N}}e^{-\rho s}w_x(X^{x,y,I}_s,Y^{y,I}_s)dW_s}_{=:M_{\tau_{R,N}}}\\&
	+\sum_{0\leq s\leq \tau_{R,N}}e^{-\rho s}\big[w(X^{x,y,I}_s,Y^{y,I}_s)-w(X^{x,y,I}_{s},Y^{y,I}_{s-})\big]+\int_{0}^{\tau_{R,N}} e^{-\rho s} w_y(X^{x,y,I}_s,Y^{y,I}_s)dI_s^c,
	\end{split}
	\end{align}
	upon noticing that $t\mapsto X^{x,y,I}_t$ is continuous almost surely for any $I\in \mathcal{I}^{\bar{y}}(y)$. Now, we find
	\begin{align*}
	w(X^{x,y,I}_s,Y^{y,I}_s)-w(X^{x,y,I}_{s},Y^{y,I}_{s-})
	=&w(X^{x,y,I}_{s},Y^{y,I}_{s-}+\Delta I_s)-w(X^{x,y,I}_{s},Y^{y,I}_{s-})\\
	=&\int_{0}^{\Delta I_s} w_y(X^{x,y,I}_{s},Y^{y,I}_{s-}+u)du,
	\end{align*}
	which substituted back into \eqref{eq14} gives the equivalence
	\begin{align*}
	&\int_{0}^{\tau_{R,N}}e^{-\rho s}X^{x,y,I}_sY^{y,I}_sds-c\int_{0}^{\tau_{R,N}}e^{-\rho s}dI_s\\
	=& w(x,y)-e^{-\rho \tau_{R,N}}w(X^{x,y,I}_{\tau_{R,N}},Y^{y,I}_{\tau_{R,N}})\\
	&+\int_{0}^{\tau_{R,N}}e^{-\rho s}\Big(\mathcal{L}^yw(X^{x,y,I}_s,Y^{y,I}_s)-\rho w(X^{x,y,I}_s,Y^{y,I}_s)+X^{x,y,I}_sY^{y,I}_s\Big)ds+M_{\tau_{R,N}}\\&
	+\sum_{0\leq s\leq \tau_{R,N}}e^{-\rho s}\int_{0}^{\Delta I_s}\big[w_y(X^{x,y,I}_{s},Y^{y,I}_{s-}+u)-c\big]du+\int_{0}^{\tau_{R,N}}e^{-\rho s}\big[ w_y(X^{x,y,I}_{s},Y^{y,I}_{s})-c\big]dI_s^c,
	\end{align*}
	by adding $\int_{0}^{\tau_{R,N}}e^{-\rho s}X^{x,y,I}_sY^{y,I}_sds-c\int_{0}^{\tau_{R,N}}e^{-\rho s}dI_s$ on both sides of \eqref{eq14}. Since $w$ satisfies \eqref{HJB1} and \eqref{eq15}, by taking expectations on both sides of the latter equation, and using that $\mathbb{E}[M_{\tau_{R,N}}]=0$, we have
	\begin{align}\label{eq2}
	\mathbb{E}\Big[\int_{0}^{\tau_{R,N}}e^{-\rho s}X^{x,y,I}_sY^{y,I}_sds-c\int_{0}^{\tau_{R,N}}e^{-\rho s}dI_s\Big]\leq w(x,y)+K\mathbb{E}\Big[e^{-\rho \tau_{R,N}}\Big(1+|X^{x,y,I}_{\tau_{R,N}}|\Big)\Big].
	\end{align}
	In order to apply the dominated convergence theorem in \eqref{eq2}, we notice on the one hand that $X^{x,y,I}_t\leq X_t^{x}$ $\mathbb{P}$-a.s. for all $t\geq 0$, and therefore that
	\begin{align*}
	X^{x,y,I}_t&=x+\int_0^t\kappa\big((\mu-\beta Y^{y,I}_t)-X^{x,y,I}_s\big)ds +\sigma W_t\geq x+\int_0^t\kappa\big(\mu-X^{x}_s\big)ds +\sigma W_t-\kappa\beta\bar{y}t\\
	&=X^{x}_t-\kappa\beta\bar{y}t\geq -|X_t^{x}|-\kappa\beta\bar{y}t,
	\end{align*} 
	where we have used that $Y^{y,I}_t\leq \bar{y}$ $\mathbb{P}$-a.s. for all $t\geq 0$.
	Also, one clearly has $X^{x,y,I}_t\leq X_t^{x}\leq |X_t^{x}|+\kappa\beta\bar{y}t$. Hence,  
	\begin{align}\label{eq12}
	|X^{x,y,I}_t|\leq |X_t^{x}|+\kappa\beta\bar{y}t.
	\end{align} 
	
	Now, we find that $\mathbb{P}$-a.s.
	\begin{align}\label{eq19}
	\begin{split}
	\bigg|\int_{0}^{\tau_{R,N}}e^{-\rho s}X^{x,y,I}_sY^{y,I}_sds-c\int_{0}^{\tau_{R,N}}e^{-\rho s}dI_s\bigg|\leq \bar{y}\int_{0}^{\infty}e^{-\rho s}\Big(|X^{x}_s|+\kappa\beta\bar{y}s\Big)ds+c\bar{y},
	\end{split}
	\end{align}
	and the first expression on the right-hand side of \eqref{eq19} is integrable by \eqref{supfinite}. On the other hand, so to take care of the expectation on the right-hand side of \eqref{eq2}, we employ again \eqref{eq12} to get for some constant $C_1>0$
	\begin{align}
	\begin{split}\label{eq11}
	\mathbb{E}\Big[e^{-\rho \tau_{R,N}} (1+|X_{\tau_{R,N}}^{x,y,I}|)\Big]&\leq C_1 \mathbb{E}\Big[e^{-\rho \tau_{R,N}}\left(1+\tau_{R,N}\right)\Big] + \mathbb{E}\Big[e^{-\frac{\rho}{2} \tau_{R,N}}\sup\limits_{t\geq 0}e^{-\frac{\rho}{2}t}|X_t^{x}|\Big]\\
	& \leq C_1 \mathbb{E}\Big[e^{-\rho \tau_{R,N}}\left(1+\tau_{R,N}\right)\Big] + \mathbb{E}\Big[e^{-{\rho} \tau_{R,N}}\Big]^{\frac{1}{2}}\mathbb{E}\Big[\sup\limits_{t\geq 0}e^{-\rho t}(X_t^{x})^2\Big]^{\frac{1}{2}},
	\end{split}
	\end{align}
	where we have used H\"older's inequality in the last step. 
	As for the last expectation in \eqref{eq11}, observe that by It\^o's formula we find
	\begin{align}
	\begin{split}\label{eq59}
	e^{-\rho t}(X^{x}_t)^2\leq x^2&+ \int_0^te^{-\rho u}\Big[\rho(X^{x}_u)^2+\sigma^2\Big]du\\&+\int_0^t2e^{-\rho u}|X^{x}_u|(\kappa(|\mu|+|X^{x}_u|))du+2\sigma\sup\limits_{t\geq 0}\bigg|\int_0^te^{-\rho u} X^{x}_u dW_u\bigg|.
	\end{split}
	\end{align}
	Then, by an application of the Burkholder-Davis-Gundy inequality (cf. Theorem 3.28 in \cite{karatzas}), we find that
	\begin{align}\label{eq43}
	\mathbb{E}\Big[\sup_{t\geq 0}\Big|\int_0^te^{-\rho u}\sigma X^{x}_u dW_u\Big|\Big]\leq C_2(1+|x|),
	\end{align}
	for some constant $C_2>0$. Then, since standard calculations show that $\mathbb{E}\big[|X_u^{x}|^q\big]\leq \tilde{C}(1+|x|^q)$ for $q\in\{1,2\}$ and some $\tilde{C}>0$, we obtain from \eqref{eq59} and \eqref{eq43} 
	\begin{align}\label{eq18}
	\mathbb{E}\Big[\sup\limits_{t\geq 0}e^{-\rho t}(X_t^{x})^2\Big]\leq C_3(1+x^2),
	\end{align} for some constant $C_3>0$, and therefore, it follows with \eqref{eq11} 
	\begin{align}\label{eq130}
	\lim\limits_{N\uparrow\infty}\lim\limits_{R\uparrow\infty}\mathbb{E}\Big[e^{-\rho \tau_{R,N}} (1+|X_{\tau_{R,N}}^{x,y,I}|)\Big]=0.
	\end{align} Hence, we can invoke the dominated convergence theorem in order to take limits as $R\rightarrow\infty$ and then as $N\rightarrow\infty$, so to get
	\begin{align}\label{eqVT}
	\mathcal{J}(x,y,I)\leq w(x,y).
	\end{align}
	
	Since $I\in\mathcal{I}^{\bar{y}}(y)$ is arbitrary, we have
	\begin{align}\label{eq58}
	V(x,y)\leq w(x,y),
	\end{align}
	which yields $V\leq w$ by arbitrariness of $(x,y)$ in $\mathbb{R}\times[0,\bar{y})$.\vspace{0.25cm}
	
	\emph{Step 2.} Let $I^\star\in\mathcal{I}^{\bar{y}}(y)$ satisfying \eqref{eq2VT} and \eqref{eq3VT}, and $\tau^\star_{R,N}:=\inf\{t\geq 0: X^{x,y,I^\star}_{t}\notin(-R,R)\}\wedge N$. Employing the same arguments as in \emph{Step 1} all the inequalities become equalities and we obtain 
	\begin{align}\label{eq17}
	&\mathbb{E}\Big[\int_{0}^{\tau_{R,N}}e^{-\rho s}X^{x,y,I^\star}_sY^{y,I^\star}_sds-c\int_{0}^{\tau_{R,N}}e^{-\rho s}dI^\star_s\Big]+\mathbb{E}\Big[e^{-\rho \tau^\star_{R,N}}w(X^{x,y,I^\star}_{\tau^\star_{R,N}},I^\star_{\tau^\star_{R,N}})\Big]= w(x,y).
	\end{align}
	Now, because $I^\star$ is admissible and upon employing \eqref{eq15} and \eqref{eq130}, we proceed as in \emph{Step 1} , and take limits as $R\uparrow\infty$ and $N\uparrow\infty$ in \eqref{eq17}, so to find $\mathcal{J}(x,y,I^\star)\geq w(x,y)$. Since clearly $V(x,y)\geq\mathcal{J}(x,y,I^\star)$, then $V(x,y)\geq w(x,y)$ for all $(x,y)\in\mathbb{R}\times[0,\bar{y})$. Hence, using \eqref{eq58} $V=w$ on $\mathbb{R}\times [0,\bar{y})$ and $I^\star$ is optimal. 
\end{proof}

\section{Constructing an Optimal Solution to the Installation Problem}
\label{sec:Sol}
In this section, we first construct a candidate value function and a candidate optimal strategy. Then, we move on by verifying their optimality. 

We make the guess that there exists an injective function $F:[0,\bar{y}]\to\mathbb{R}$, called the \emph{free boundary} which separates the \emph{waiting region} $\mathbb{W}$ and the \emph{installation region} $\mathbb{I}$, such that
\begin{align}\label{WR}
\mathbb{W}&=\{(x,y)\in\mathbb{R}\times[0,\bar{y}):\,x< F(y)\},\\\label{IR}
\mathbb{I}&=\{(x,y)\in\mathbb{R}\times[0,\bar{y}):\,x\geq F(y)\}.
\end{align} 

For all $(x,y)\in\mathbb{W}$, the candidate value function $w$ should satisfy (cf. \eqref{WROld})
\begin{align}\label{cond1}
\mathcal{L}^yw(x,y)-\rho w(x,y)+xy=0.
\end{align}
Recall \eqref{PS}. It is straightforward to check that a particular solution to \eqref{cond1} is given by the function $R$. Moreover, the homogeneous differential equation 
\begin{align}\label{odehom}
\mathcal{L}^yw(x,y)-\rho w(x,y)=0,
\end{align} admits two fundamental strictly positive solutions (see pp. 18-19 of \cite{Borodin}). These are given by $\phi(x+\beta y)$ and $\psi(x+\beta y)$, with $\phi(\cdot)$ strictly decreasing and $\psi(\cdot)$ strictly increasing, cf. Lemma \ref{Properties}-(1),(5). 
Therefore our candidate value function $w$ takes the form
\begin{align}
w(x,y)=A(y)\psi(x+\beta y)+B(y)\phi(x+\beta y)+R(x,y),\quad(x,y)\in\mathbb{W},
\end{align} for some functions $A,B:[0,\bar{y}]\mapsto\mathbb{R}$ to be found. Notice that, for $y\geq 0$ be given and fixed, $\phi(x+\beta y)$ grows to $+\infty$ exponentially fast whenever $x\downarrow-\infty$, cf. Appendix 1 in \cite{Borodin}. In light of the linear growth of $V$, see Proposition \ref{GrowthV}, and the structure of the waiting region $\mathbb{W}$, cf. \eqref{WR}, we must then have $B(y)=0$ for all $y\in[0,\bar{y}]$. 
Thus, we conjecture that 
\begin{align}\label{Repw}
w(x,y)=A(y)\psi\left(x+{\beta}y\right)+R(x,y),\quad\text{for $(x,y)\in\mathbb{W}$}.
\end{align}

We move on to derive equations that characterize the function $A$ and the free boundary $F$. With reference to \eqref{IROld}, for all $(x,y)\in\mathbb{I}$, $w$ should instead satisfy
\begin{align}\label{cond2}
w_y(x,y)-c=0,
\end{align}
implying 
\begin{align}\label{cond3}
w_{yx}(x,y)=0.
\end{align}
Now, we impose the so-called \emph{Smooth Fit} condition, i.e. we suppose that $w\in C^{2,1}(\mathbb{R}\times[0,\bar{y}])$, and therefore by \eqref{Repw},\eqref{cond2} and \eqref{cond3}, $w$ should satisfy
\begin{align}\label{cond4}
A'(y)\psi\big(F(y)+{\beta}y\big)+{\beta}A(y)\psi'\big(F(y)+{\beta}y\big)+R_y(F(y),y)-c=0,
\end{align}
and 
\begin{align}\label{cond5}
A'(y)\psi'\big(F(y)+{\beta}y\big)+{\beta}A(y)\psi''\big(F(y)+{\beta}y\big)+R_{yx}(F(y),y)=0.
\end{align}
Notice that the derivatives of $R$ can be easily obtained from \eqref{PS}, which gives
$$R_y(x,y)=\frac{x}{\rho+\kappa}+\frac{\mu\kappa}{\rho(\rho+\kappa)}-\frac{2\kappa\beta y}{\rho(\rho+\kappa)},\quad\text{and}\quad R_{xy}(x,y)=\left(\rho+\kappa\right)^{-1}.$$

The following lemma provides essential properties of the function $A$ and a lower bound for $F$ that are needed for results of Section \ref{sec:ExFB} and Section \ref{sec:Verification}. Its proof can be found in the appendix.
\begin{lemma}\label{PropA}
	The function $A$ is strictly positive and strictly decreasing. Moreover, $A$ admits the representation
	\begin{align}\label{Ayneu}
	A(y)=\left(\beta\rho(\rho+\kappa)\right)^{-1}\times\frac{(\rho+\kappa)\Big(c\rho+\frac{\kappa\beta}{\rho+\kappa}y-F(y)\Big)\psi'(F(y)+{\beta}y)+\frac{\sigma^2}{2}\psi''(F(y)+{\beta}y)}{\psi'(F(y)+{\beta}y)^2-\psi''(F(y)+{\beta}y)\psi(F(y)+{\beta}y)},
	\end{align} and we have 
	\begin{align}\label{crho}
	F(y)\geq c\rho+\frac{\kappa\beta}{\rho+\kappa}y\geq c\rho,\quad\text{for all }y\in[0,\bar{y}].
	\end{align}
\end{lemma}

\subsection{The Free Boundary: Existence and Characterization}
\label{sec:ExFB}
For the sake of simplicity, we introduce the function $\tilde{F}$ for a substitution, that is 
\begin{align}\label{Substitution}
\tilde{F}(y)=F(y)+{\beta}y.
\end{align}
We aim to prove the existence and a monotonicity property of $\tilde{F}$, so to draw the implications for $F$ after.
We have 
$$ R_y(F(y),y)  = \frac{\rho F(y) + \mu\kappa - 2\kappa\beta y}{\rho(\rho+\kappa)} =  \frac{\mu\kappa  +\rho \tilde{F}(y) - \beta (\rho + 2\kappa) y}{\rho(\rho+\kappa)} = \tilde R(\tilde F(y),y),$$ 
where $\tilde R:\mathbb{R}^2\mapsto\mathbb{R}$ is defined as
$$\tilde R(x,y) := \frac{\mu\kappa  +\rho x - \beta (\rho + 2\kappa) y}{\rho(\rho+\kappa)}.$$ 
Notice that
$$ \tilde R_x(\tilde{F}(y),y)=(\rho+\kappa)^{-1}=R_{yx}(F(y),y). $$ 
{From now on, we will often use the functions $Q_k:\mathbb{R}\mapsto\mathbb{R}$, $k\in\mathbb{N}_0$, and their first derivatives, given by
	\begin{align}\label{Q}
	\begin{split}
	Q_k(z) := & \psi^{(k)}(z)\psi^{(k+2)}(z)-\psi^{(k+1)}(z)^2, \\
	Q'_k(z) = & \psi^{(k)}(z)\psi^{(k+3)}(z) - \psi^{(k+1)}(z) \psi^{(k+2)}(z).
	\end{split}
	\end{align}
Substituting $\tilde{F}$ for $F$ in both \eqref{cond4} and \eqref{cond5}, and solving for $A$ and $A'$, gives
\begin{align}\label{Aytilde}
A(y)=\beta^{-1}\times\frac{\psi'(\tilde{F}(y))\Big(c-\tilde{R}(\tilde{F}(y),y)\Big)+\left(\rho+\kappa\right)^{-1}\psi(\tilde{F}(y))}{-Q_0(\tilde{F}(y))},
\end{align}
and 
\begin{align}\label{Aprimey}
A'(y)=\frac{\psi''(\tilde{F}(y))\Big(c-\tilde{R}(\tilde{F}(y),y)\Big)+\left(\rho+\kappa\right)^{-1}\psi'(\tilde{F}(y))}{Q_0(\tilde{F}(y))}.
\end{align} Lemma \ref{Properties}-(3) ensures that $Q_k$ is strictly positive for all $k\in\mathbb{N}_0$, and therefore the denominator on the right-hand side of both \eqref{Aytilde} and \eqref{Aprimey} is nonzero.}

In light of the boundary condition $w(x,\bar{y})=R(x,\bar{y})$, cf. Theorem \ref{VerificationTheorem}, the function $A$ should satisfy
\begin{align}\label{BoundCond}
A(\bar{y})=0.
\end{align}
Due to $\eqref{Aytilde}$ and \eqref{BoundCond}, we must have that there exists a point $\tilde{{x}}=\tilde{F}(\bar{y})\in\mathbb{R}$ solving $H(x)=0$, where $H:\mathbb{R}\mapsto\mathbb{R}$ is defined as
\begin{align}\label{hxinf}
H(x):=\psi'(x)\left(c-\tilde{R}(x,\bar{y})\right)+\left(\rho+\kappa\right)^{-1}\psi(x).
\end{align}

\begin{lemma}\label{UniqueSol}
	There exists a unique solution $\tilde{x}\in\mathbb{R}$ to the equation $H(x)=0$.
\end{lemma}
\begin{proof}
	We rewrite $H(x):=-\left(\rho+\kappa\right)^{-1}\left(\psi'(x)\left((\rho+\kappa)\tilde{R}(x,\bar{y})-c(\rho+\kappa)\right)-\psi(x)\right).$ Now, the proof is a slight modification of the proof of Lemma 4.4 in \cite{Koch2} upon adjusting the cost factor in \cite{Koch2} by $c(\rho+\kappa)-\frac{\mu\kappa- \beta (\rho + 2\kappa) \bar{y}}{\rho}$.

\end{proof}
Differentiating \eqref{Aytilde}, we find
\begin{align}\label{Ay'}
A'(y)=&(\beta(\rho+\kappa))^{-1}\times\frac{P(y,\tilde{F}(y),\tilde{F}'(y))}{Q_0(\tilde{F}(y))^2},
\end{align}
where $P:\mathbb{R}^3\mapsto\mathbb{R}$ is given by
\begin{align*}
\begin{split}
P(y,z,w):=&w(\rho+\kappa)\Big(c-\tilde{R}(z,y)\Big)\psi(z)\Big(\psi'''(z)\psi'(z)-\psi''(z)^2\Big)\\
&+\frac{\beta(\rho+2\kappa)}{\rho}\psi'(z)\Big(\psi'(z)^2-\psi(z)\psi''(z)\Big)-w\psi(z)\Big(\psi'(z)\psi''(z)-\psi(z)\psi'''(z)\Big)\\
=&-\frac{\beta(\rho+2\kappa)}{\rho}\psi'(z)Q_0(z)+w D(y,z), 
\end{split}
\end{align*}
with $D:\mathbb{R}^2\mapsto\mathbb{R}$ defined as
\begin{align}\label{DenOld}
\begin{split}
D(y,z) 
{ = \psi(z)\bigg[ } & { (\rho + \kappa) (c -\tilde{R}(z,y)) Q_1(z) + Q_0'(z) \bigg]. }
\end{split}
\end{align}
Now, equating both expressions \eqref{Aprimey} and \eqref{Ay'}, we get
\begin{align}\label{eq6}
P(y,\tilde{F}(y),\tilde{F}'(y))=\beta Q_0(\tilde{F}(y))\Big((\rho+\kappa)\left(c-\tilde{R}(\tilde{F}(y),y)\right)\psi''(\tilde{F}(y))+\psi'(\tilde{F}(y))\Big).
\end{align}
Letting $N:\mathbb{R}^2\mapsto\mathbb{R}$ be 
such that
\begin{align}\label{NomOld}
\begin{split}
N(y,z)=Q_0(z)\Bigg(&\frac{\rho+2\kappa}{\rho}\psi'(z)+\bigg((\rho+\kappa)\left(c-\tilde{R}(z,y)\right)\psi''(z)+\psi'(z)\bigg)\Bigg),
\end{split}
\end{align}
we obtain from \eqref{eq6} the ODE
\begin{align}\label{ODE}
\begin{split}
\tilde{F}'(y)=\mathcal{G}(y,\tilde{F}(y)),
\end{split}
\end{align}
with boundary condition $\tilde{F}(\bar{y})=\tilde{x}$, cf. Lemma \ref{UniqueSol}, and where $\mathcal{G}:(\mathbb{R}\times\mathbb{R})\setminus\{(y,z)\in\mathbb{R}^2:D(y,z)=0\}\mapsto\mathbb{R}$ is such that
\begin{align}\label{ODEalt}
\begin{split}
\mathcal{G}(y,z)=\beta\times\frac{N(y,z)}{D(y,z)}.
\end{split}
\end{align}

The next goal is to prove that the ODE \eqref{ODE} admits a unique solution $\tilde{F}$ on $[0,\bar{y}]$ such that $\tilde{F}'(y)\geq\beta$. As a preliminary result we show that the previous property holds at $\bar{y}$, that is $\mathcal{G}(\bar{y},\tilde{x})>\beta$.
\begin{lemma}\label{lemma:ybar}
	For any $\bar{y}>0$, we have
	$D(\bar{y},\tilde{F}(\bar{y}))>0,$ and it holds 
	\begin{align}\label{F'ybar}
	\tilde{F}'(\bar{y})>\beta.
	\end{align}
\end{lemma} 

\begin{proof}
	Recall the function $H$ from \eqref{hxinf} which is such that $H(\tilde{F}(\bar{y}))=0$. Therefore $\bar{y}$ satisfies 
	\begin{align}\label{eq71}
	(\rho+\kappa)\left(c-\tilde{R}(\tilde{F}(\bar{y}),\bar{y})\right)=-\frac{\psi(\tilde{F}(\bar{y}))}{\psi'(\tilde{F}(\bar{y}))}.
	\end{align}
	We get from \eqref{DenOld} and \eqref{eq71} that
	\begin{align}\label{eq36}
	D(\bar{y},\tilde{F}(\bar{y}))=\frac{Q_0(\tilde{F}(\bar{y}))\psi(\tilde{F}(\bar{y}))\psi''(\tilde{F}(\bar{y}))}{\psi'(\tilde{F}(\bar{y}))}>0,
	\end{align}
	upon recalling that $Q_0 > 0$. Now, Lemma \ref{Implications} implies $N(\bar{y},\tilde{F}(\bar{y}))-D(\bar{y},\tilde{F}(\bar{y}))>0$. 
	Hence, we find
	\begin{align}
	\tilde{F}'(\bar{y})=\mathcal{G}(\bar{y},\tilde{F}(\bar{y}))=\beta\times\frac{N(\bar{y},\tilde{F}(\bar{y}))}{D(\bar{y},\tilde{F}(\bar{y}))}>\beta.
	\end{align}
\end{proof}



Now, we state the main result in this subsection. It guarantees the existence and uniqueness of a solution $\tilde{F}$ on $[0,\bar{y}]$ of \eqref{ODE} which is such that $\tilde{F}'(y)>\beta$ for all $y\in[0,\bar{y}]$. Its proof can be found in the appendix.
\begin{proposition}\label{Prop:SolODE}
	For any $\bar{y}>0$, there exists a unique solution $\tilde{F}$ on $[0,\bar{y}]$ of the ODE \eqref{ODE} with boundary condition $\tilde F(\bar{y})=\tilde{x}$. Moreover, $$\tilde{F}'(y)\geq\beta,\quad \text{for all $y\in[0,\bar{y}]$}.$$	
\end{proposition}
\begin{corollary} \label{Cor:4.5}
	The free boundary $F$ as in \eqref{WR} and \eqref{IR} is well defined. Moreover, it is strictly increasing and given by $$F(y)=\tilde{F}(y)-\beta y,\quad\text{for all }y\in[0,\bar{y}].$$
%
%
	
\end{corollary}
\begin{proof}
	The existence and uniqueness is an implication of Proposition \ref{Prop:SolODE}. It also ensures that $$F'(y)=\tilde{F}'(y)-\beta>0,\quad\text{for all }y\in[0,\bar{y}].$$
\end{proof}

\subsection{The Optimal Strategy and the Value Function: Verification}
\label{sec:Verification}
In the following, the initial price level at which the company starts to install solar panels is denoted by $x_0:=F(0)$, and we define $\bar{x}:=\tilde{x}-\beta\bar{y}$. Since $F$ is strictly increasing, its inverse function exists on $[x_0,\bar{x}]$ and is denoted by $F^{-1}$. 

We divide the (candidate) \emph{installation region} $\mathbb{I}$ into $$\mathbb{I}_1:=\{(x,y)\in\mathbb{R}\times[0,\bar{y}):\,x\in[F(y),\bar{x})\},$$ and $$\mathbb{I}_2:=\{(x,y)\in\mathbb{R}\times[0,\bar{y}):\,x\geq \bar{x}\}.$$ An optimal installation strategy can be described as follows: in $\mathbb{W}$ (cf. \eqref{WR}), that is if the current price $x$ is sufficiently low such that $x<F(y)$, then the company does not increase the level of installed power. Whenever the price crosses $F(y)$, then the company makes infinitesimal installations so to keep the state process $(X,Y)$ inside $\overline{\mathbb{W}}$. Conversely, if the current price $x$ is sufficiently large such that $x\geq F(y)$ (i.e. in $\mathbb{I}$, cf. \eqref{IR}), then the company makes an instantaneous lump sum installation. In particular, on the one hand, whenever the maximum level of installed power $\bar{y}$, that the firm is able to reach, is sufficiently high (that is $(x,y)\in\mathbb{I}_1$),  then the company pushes the state process $(X,Y)$ immediately to the locus of points $\{(x,y)\in\mathbb{R}\times[0,\bar{y}]:\, x=F(y)\}$ in direction $(0,1)$, so to increase the level of installed power by $F^{-1}(x)-y$ units. The associated payoff to this action is then the difference of the continuation value starting from the new state $(x,F^{-1}(x))$ and the costs associated to the installation of additional solar panels, that is $c(F^{-1}(x)-y)$. On the other hand, whenever the firm has to restrict its actions due to the upper bound 
$\bar{y}$ (that is $(x,y)\in\mathbb{I}_2$), then the company immediately installs the maximum number of panels, so to increase the level of installed power up to $\bar{y}$ units, and the associated payoff to such a strategy is $R(x,\bar{y})-c(\bar{y}-y)$.

In light of the previous discussion, we now define our candidate value function $w:\mathbb{R}\times[0,\bar{y}]\mapsto\mathbb{R}$ as
{\begin{align}\label{CaFnc}
w(x,y)=
\begin{cases}
A(y)\psi\big(x+{\beta}y\big)+R(x,y),\vspace{1.5mm}\quad &\text{if } x\in \mathbb{W}\cup\left((-\infty,\bar{x})\times\{\bar{y}\}\right),\\
A(F^{-1}(x))\psi\big(x+{\beta}F^{-1}(x)\big)+R(x,F^{-1}(x))\vspace{1,5mm}\\
-c(F^{-1}(x)-y),&\text{if } (x,y)\in\mathbb{I}_1,\\
R(x,\bar{y})-c(\bar{y}-y),\vspace{1.5mm} &\text{if } (x,y)\in\mathbb{I}_2\cup\left([\bar{x},\infty)\times\{\bar{y}\}\right).
\end{cases}
\end{align} }

The next two results verify that $w$ is a classical solution to the HJB equation \eqref{HJB1}.
\begin{lemma}\label{Smooth}
	The function $w$ is $C^{2,1}(\mathbb{R}\times[0,\bar{y}])$.
\end{lemma}

\begin{proof}
	In the following, we denote by $\I(\cdot)$ the interior of a set. Clearly, by \eqref{CaFnc} it holds for all $(x,y)\in\I(\mathbb{W})$ that
	\begin{align}\label{DerWait}
	w_x(x,y)&=A(y)\psi^\prime(x+{\beta}y)+R_x(x,y),\\
	w_{xx}(x,y)&=A(y)\psi^{\prime\prime}(x+{\beta}y),\\\label{DerWait3}
	w_y(x,y)&=A^\prime(y)\psi(x+{\beta}y)+{\beta}A(y)\psi'(x+{\beta}y)+R_y(x,y),
	\end{align}
	and for all $(x,y)\in\I(\mathbb{I}_2)$ we have
	\begin{align}\label{DerI2}
	w_x(x,y)=R_x(x,\bar{y}),\quad
	w_{xx}(x,y)=0,\quad
	w_y(x,y)=c.
	\end{align}
	To evaluate $w_x,w_{xx}$ and $w_y$ inside $\mathbb{I}_1$, we need some more work. 
	We find for all $(x,y)\in\I(\mathbb{I}_1)$
	\begin{align}
	\begin{split}\label{eq1}
	w_x(x,y)=&A(F^{-1}(x))\psi'\big(x+{\beta}F^{-1}(x)\big)+R_x\big(x,F^{-1}(x)\big)\\
	&+(F^{-1})'(x)\bigg[A'\big(F^{-1}(x)\big)\psi\big(x+{\beta}F^{-1}(x)\big)+{\beta}A\big(F^{-1}(x)\big)\psi'\big(x+{\beta}F^{-1}(x)\big)\\
	&+R_y\big(x,F^{-1}(x)\big)-c\bigg],\\
	=&A(F^{-1}(x))\psi'\big(x+{\beta}F^{-1}(x)\big)+R_x\big(x,F^{-1}(x)\big),
	\end{split}\\
	\begin{split}\label{eq3}
	w_{xx}(x,y)=&A(F^{-1}(x))\psi^{\prime\prime}(x+{\beta}F^{-1}(x))+(F^{-1})'(x)\Big[A'(F^{-1}(x))\psi'\big(x+{\beta}F^{-1}(x)\big)\\
	&+{\beta}A(F^{-1}(x))\psi''\big(x+{\beta}F^{-1}(x)\big)+R_{yx}(x,F^{-1}(x))\Big]\\
	=&A(F^{-1}(x))\psi^{\prime\prime}(x+{\beta}F^{-1}(x)),
	\end{split}\\
	w_y(x,y)=&c,
	\end{align}
	where we have used \eqref{cond4} in \eqref{eq1}, and \eqref{cond5} in \eqref{eq3}. Notice that the functions $A,\, F^{-1},\,\psi,$ $\psi',$ $R_y$ and $R_x$ are continuous. {The previous equations and \eqref{cond4} easily provide the continuity of the derivatives on $\mathbb{R}\times\{\bar{y}\}$.} Letting $(x_n,y_n)_{n}\subset\mathbb{I}_1$ be any sequence converging to $(F(y),y)$, $y\in[0,\bar{y})$, we find the required continuity results along $\overline{\mathbb{W}}\cap\overline{\mathbb{I}}_1$ upon employing \eqref{cond4}. Moreover, the boundary condition, cf. \eqref{BoundCond}, ensures the continuity of $w_x$ and $w_{xx}$ along $\overline{\mathbb{I}}_1\cap\overline{\mathbb{I}}_2$, and we clearly have the continuity of $w_y$ along $\overline{\mathbb{I}}_1\cap\overline{\mathbb{I}}_2$. .
\end{proof}

\begin{proposition}\label{HJB}
	The function $w$ from \eqref{CaFnc} is a $C^{2,1}(\mathbb{R}\times[0,\bar{y}])$ solution to \begin{align}
	\max\Big\{\mathcal{L}^yw(x,y)-\rho w(x,y)+xy,w_y(x,y)-c\Big\}=0,\quad\text{for all $(x,y)\in\mathbb{R}\times[0,\bar{y})$,}
	\end{align}
	 such that $w(x,\bar{y})=R(x,\bar{y})$.
	
\end{proposition}

\begin{proof}
	{Lemma \ref{Smooth} guarantees the claimed regularity of $w$. Moreover, from \eqref{CaFnc} we see that $w(x,\bar{y})=R(x,\bar{y})$ since $A(\bar{y})=0$,} and by construction, we clearly have $\mathcal{L}^yw(x,y)-\rho w(x,y)+xy=0$ for all $(x,y)\in\mathbb{W}$, and $w_y(x,y)-c=0$ for all $(x,y)\in\mathbb{I}_1\cup\mathbb{I}_2$. We prove the inequalities $\mathcal{L}^yw(x,y)-\rho w(x,y)+xy\leq 0$ for all $(x,y)\in\mathbb{I}$, and $w_y(x,y)-c\leq 0$ for all $(x,y)\in\mathbb{W}$, in the following three steps separately. {It is worth to bear in mind that $R_x(x,y)=\frac{y}{\rho+\kappa}$ by \eqref{PS}.}
	
	\vspace{0.25cm}
	\emph{Step 1.} Let $(x,y)\in\mathbb{I}_1$ be fixed. From the second line of \eqref{CaFnc}, \eqref{eq1} and \eqref{eq3}, we find
	\begin{align}\label{eq42}
	\begin{split}
	&\mathcal{L}^yw(x,y)-\rho w(x,y)+xy\\
	=&\mathcal{L}^{F^{-1}(x)}w(x,F^{-1}(x))-\rho w(x,F^{-1}(x))+xF^{-1}(x)\\
	&+{\kappa\beta}w_x(x,F^{-1}(x))(F^{-1}(x)-y)+(c\rho-x)(F^{-1}(x)-y)\\
	=&(F^{-1}(x)-y)\left(c\rho+{\kappa\beta}w_x(x,F^{-1}(x))-x\right),
	\end{split}
	\end{align}
	where we have employed that $w(x,F^{-1}(x))$ solves $$\mathcal{L}^{F^{-1}(x)}w(x,F^{-1}(x))-\rho w(x,F^{-1}(x))+xF^{-1}(x)=0.$$ For any $(x,y)\in\mathbb{I}_1$, we have $x\geq F(y)$ implying $F^{-1}(x)\geq y$ because $F$, and hence $F^{-1}$, is strictly increasing, cf. Corollary \ref{Cor:4.5}. Thus, in order to show that \eqref{eq42} is negative on $\mathbb{I}_1$, it suffices to prove that the function 
	\begin{align}\label{FuncZ}
	Z(x,F^{-1}(x)):=c\rho+{\kappa\beta}w_x(x,F^{-1}(x))-x,
	\end{align} is negative for any $x\in[x_0,\bar{x}]$. This can be accomplished in the same way as in \emph{Step 2} in the proof of Proposition \ref{Prop:SolODE}. Due to the regularity of $w$, we can use \eqref{eq1}, and the fact that $A(F^{-1}(\bar{x}))=A(\bar{y})=0$, to obtain 
	\begin{align}\label{eq50}
	Z(\bar{x},F^{-1}(\bar{x}))=c\rho+R_x(\bar{x},\bar{y})-\bar{x}< 0,
	\end{align}
	where the inequality holds by \eqref{crho} with $y=\bar{y}$. Taking the total derivative of $Z(x,F^{-1}(x))$ with respect to $x$ gives
	\begin{align}\label{DerZ}
	\begin{split}
	\frac{dZ(x,F^{-1}(x))}{dx}=&\kappa\beta w_{xx}(x,F^{-1}(x))-1
	=\kappa\beta A(F^{-1}(x))\psi''(x+\beta F^{-1}(x))-1\\
	=&\left[\rho\left(\psi(x+\beta F^{-1}(x))\psi''(x+\beta F^{-1}(x))-\psi'(x+\beta F^{-1}(x))^2\right)\right]^{-1}\\
	&\times\bigg[\rho\left(\psi'(x+\beta F^{-1}(x))^2-\psi(x+\beta F^{-1}(x))\psi''(x+\beta F^{-1}(x))\right)\\
	&-\kappa\psi'(x+\beta F^{-1}(x))\psi''(x+\beta F^{-1}(x))\left(c\rho+\frac{\kappa\beta}{\rho+\kappa}F^{-1}(x)-x\right)\\
	&-\frac{\sigma^2}{2}\kappa\psi''(x+\beta F^{-1}(x))^2R_{xy}(x,F^{-1}(x))\bigg],
	\end{split}
	\end{align}
	where we have employed: $w_{xy}(x,F^{-1}(x))=0$, cf. \eqref{cond3}, for the first equality, and \eqref{Ayneu} for the last equality (after rearranging terms).\\
	Now, suppose that there exists a point $x^\star\in[x_0,\bar{x})$ such that $Z(x^\star,F^{-1}(x^\star))=0$. It follows from \eqref{FuncZ}, together with \eqref{Ayneu} and \eqref{eq1}, that $(x^\star,F^{-1}(x^\star))$ satisfies
	\begin{align}\label{CondNull}
	\begin{split}
	&c\rho+\frac{\kappa\beta}{\rho+\kappa}F^{-1}(x^\star)-x^\star\\=&\frac{-\frac{\sigma^2}{2}\kappa\psi'(x^\star+\beta F^{-1}(x^\star))\psi''(x^\star+\beta F^{-1}(x^\star))R_{xy}(x^\star,F^{-1}(x^\star))}{(\rho+\kappa)\psi'(x^\star+\beta F^{-1}(x^\star))^2-\rho\psi(x^\star+\beta F^{-1}(x^\star))\psi''(x^\star+\beta F^{-1}(x^\star))}.
	\end{split}
	\end{align}
	Then, exploiting the latter, one can find with \eqref{DerZ} that
	\begin{align}\label{eq51}
	\begin{split}
	\frac{dZ(x,F^{-1}(x))}{dx}\bigg|_{x=x^\star}=&\frac{\sigma^2}{2}Q_1(x^\star+\beta F^{-1}(x^\star))^{-1}Q_2(x^\star+\beta F^{-1}(x^\star))
	>0,
	\end{split}
	\end{align}
	after using \eqref{Property1} with $k=0,1,2$, and some simple algebra. We conclude from both \eqref{eq50} and \eqref{eq51} that there cannot exist a point $x^\star\in[x_0,\bar{x})$ such that $Z(x^\star,F^{-1}(x^\star))=0$. Therefore, we have $\mathcal{L}^yw(x,y)-\rho w(x,y)+xy\leq 0$ for all $(x,y)\in\mathbb{I}_1$.

	\vspace{0.25cm}
	\emph{Step 2.} For all $(x,y)\in\mathbb{I}_2$ we find from the third line of \eqref{CaFnc} and \eqref{DerI2}
	\begin{align*}
	&\mathcal{L}^yw(x,y)-\rho w(x,y)+xy\\
	&=\mathcal{L}^{\bar{y}}R(x,\bar{y})-\rho R(x,\bar{y})+x\bar{y}+{\kappa\beta}R_x(x,\bar{y})(\bar{y}-y)+(c\rho-x)(\bar{y}-y)\\
	&=(\bar{y}-y)\left(\frac{\kappa\beta}{\rho+\kappa}\bar{y}+c\rho-x\right)\leq (\bar{y}-y)\left(\frac{\kappa\beta}{\rho+\kappa}\bar{y}+c\rho-\bar{x}\right)\leq 0,
	\end{align*}
	where we have used that $R(x,\bar{y})$ solves $\mathcal{L}^{\bar{y}}R(x,\bar{y})-\rho R(x,\bar{y})+x\bar{y}=0$ for the second equality, $x\geq\bar{x}$ for any $(x,y)\in\mathbb{I}_2$ for the first inequality, and \eqref{crho} with $y=\bar{y}$ and $F(\bar{y})=\bar{x}$ for the last inequality.
	
	\vspace{0.25cm}
	\emph{Step 3.} Let $(x,y)\in\mathbb{W}$ be fixed. We define $$S(x,y):=w_y(x,y)-c=A'(y)\psi(x+{\beta}y)+{\beta}A(y)\psi'(x+{\beta}y)+R_y(x,y)-c,$$
	where the last equality holds true by \eqref{DerWait3}. From \eqref{cond4} we clearly have $S(F(y),y)=0$. Hence, it suffices to show that $S_x(x,y)\geq 0$ because $x<F(y)$ for all $(x,y)\in\mathbb{W}$. Computing the derivative of $S$ with respect to $x$ gives
	$$S_x(x,y)=A'(y)\psi'(x+{\beta}y)+{\beta}A(y)\psi''(x+{\beta}y)+R_{xy}(x,y),$$
	and from \eqref{cond5} we observe that $S_x(F(y),y)=0$. Moreover, we have $$S_{xx}(x,y)=A'(y)\psi''(x+{\beta}y)+{\beta}A(y)\psi'''(x+{\beta}y).$$
	Recall \eqref{Substitution} and \eqref{DenOld}. Lemma \ref{lemma:ybar} and Proposition \ref{Prop:SolODE} imply that
	\begin{align}\label{Dpos}
	D(y,F(y)+\beta y)>0,\quad\text{for all $y\in[0,\bar{y}]$}.
	\end{align} Now, exploiting \eqref{Aytilde} and \eqref{Aprimey}, we find 
	\begin{align}\label{eq52}
	\begin{split}
	&S_{xx}(F(y),y)\\
	&=-\left[(\rho+\kappa)\psi({F}(y)+\beta y)Q_0({F}(y)+\beta y)\right]^{-1}{D}(y,{F}(y)+\beta y)<0,\quad\text{for all }y\in[0,\bar{y}],
	\end{split}
	\end{align} 
	where the inequality is due to \eqref{Dpos} and the fact that $Q_0$ is (strictly) positive. 
	Since $\frac{\psi'''(\cdot)}{\psi''(\cdot)}$ is increasing by Lemma \ref{Properties}-(3), and $A(y)$ is positive for all $y\in[0,\bar{y}]$ by Lemma \ref{PropA}, we have for all $x\leq F(y)$
	$$A'(y)+\frac{\psi'''(x+{\beta}y)}{\psi''(x+{\beta}y)}\beta A(y)<A'(y)+\frac{\psi'''(F(y)+{\beta}y)}{\psi''(F(y)+{\beta}y)}\beta A(y)<0,$$ where we have employed \eqref{eq52} for the last inequality. Thus, we have $S_{xx}(x,y)<0$, and therefore $S_x(x,y)>0$ for all $(x,y)\in\mathbb{W}$. This completes the proof.
\end{proof}

We conclude that $w$ identifies with the value function.
\begin{theorem} \label{Th:4.8}
	Recall $w$ from \eqref{CaFnc} and let $\Delta:=(\bar{y}-y)\mathds{1}_{\{x\geq \bar{x}\}}+(F^{-1}(x)-y)\mathds{1}_{\{\bar{x}>x>F(y)\}},$ $\tau:=\inf\{t\geq 0:K_t=\bar{y}-(y+\Delta)\}$, and $(X,K)$ defined on $[0,\tau]$ such that
	\begin{align}\label{SDE2}
	\begin{split}
	X_t&\leq F(y+\Delta+K_t),\\
	dX_t&=\kappa\Big(\left(\mu-{\beta}(y+\Delta+K_t)\right)-X_t\Big)dt+\sigma dW_t,\\
	dK_t&=\mathds{1}_{\{X_t=F(y+\Delta+K_t)\}}dK_t,
	\end{split}
	\end{align}
	with increasing $K$, and starting point $(X_0,K_0)=(x,0)$. Then, the function $w$ identifies with the value function $V$ from \eqref{ValueFnc}, and the optimal installation strategy, denoted by $I^\star$, is given by
	\begin{align}\label{Istar}
	\begin{cases}
	I^\star_{0-}=0\\
	I^\star_t=
	\begin{cases}
	\Delta+K_t,\quad t\in[0,\tau),\\
	\Delta+K_\tau,\quad t\geq\tau.
	\end{cases}
	\end{cases}
	\end{align} 
\end{theorem}
\
\begin{proof}
	To prove the claim, we aim at applying Theorem \ref{VerificationTheorem}. We already know that $w\in C^{2,1}(\mathbb{R}\times[0,\bar{y}])$ is a solution to the HJB equation \eqref{HJB1} by  Proposition \ref{HJB}. Moreover, the function $w$ satisfies the growth condition in \eqref{eq15} upon exploiting the facts that $A$ is continuous, $\psi$ is continuous and increasing, and $|R(x,y)|\leq K\big(1+|x|\big)$ for any $y\in[0,\bar{y}]$ and some constant $K>0$.

In a next step, we show the existence of $(X,K)$ satisfying the stochastic differential equation \eqref{SDE2}. To do so, we borrow ideas from \cite{Chiarolla2}, cf. Section 5 therein. We let $\mathbb{Q}$ be a probability measure on a filtered probability space $(\Omega, \tilde{\mathcal{F}},(\tilde{\mathcal{F}}_t)_{t\geq 0})$ with a filtration $(\tilde{\mathcal{F}}_t)_{t\geq 0}$ satisfying the usual conditions, and $B$ be a $(\tilde{\mathcal{F}}_t)_{t\geq 0}$-Brownian motion under $\mathbb{Q}$. Define the processes $({X},{K})$ such that 
	\begin{align}
	d{X}_t&=\kappa\Big(\left(\mu-{\beta}(y+\Delta)\right)-{X}_t\Big)dt+\sigma dB_t,\\
	{K}_t&=\min\Big\{\sup\limits_{0\leq s\leq t}\{\bar{F}^{-1}({X}_s)\},\bar{y}-(y+\Delta)\Big\},
	\end{align} with starting point $({X}_0,{K}_0)=(x,0)$, and where $\bar{F}^{-1}$ is such that 
	\begin{align}
	\bar{F}^{-1}(x):=
	\begin{cases}
	0,\quad&\text{if }x<x_0,\\
	F^{-1}(x),\quad&\text{if }x\in[x_0,\bar{x}],\\
	\bar{y},\quad&\text{if }x>\bar{x}.\\
	\end{cases}
	\end{align} Notice that the pair $({X},{K})$ satisfies
	\begin{align*}
	\begin{split}
	{X}_t&\leq F(y+\Delta+{K}_t),\\
	d{K}_t&=\mathds{1}_{\{{X}_t=F(y+\Delta+{K}_t)\}}d{K}_t,
	\end{split}
	\end{align*} 
	for any $t\leq\tau$. Since ${K}$ is increasing and $K_t\leq\bar{y}-(y+\Delta)$ for any $t\leq\tau$, we apply Girsanov's Theorem (cf. Section 3.5 in \cite{karatzas}), so to obtain an equivalent probability measure $\mathbb{P}$ with respect to $\mathbb{Q}$ such that for any $T>0$ $$\frac{d\mathbb{P}}{d\mathbb{Q}}\bigg|_{\mathcal{F^B}_T}=\exp\left(-\int_0^T\frac{\kappa\beta}{\sigma} {K}_sdB_s-\frac{1}{2}\int_0^T\left(\frac{\kappa\beta}{\sigma} {K}_s\right)^2ds\right),$$ and $$W_t=B_t+\int_{0}^{t}\frac{\kappa\beta}{\sigma}{K}_sds,$$ is a standard Brownian motion on $(\Omega, \mathcal{F^B},(\mathcal{F^B}_t)_{t\geq 0},\mathbb{P})$, where $(\mathcal{F^B}_t)_{t\geq 0}$ is the $\sigma$-algebra generated by $B$, and $\mathcal{F^B}=\mathcal{F^B}_\infty$. The pair $(X,K)$ constructed in this way is a weak solution to \eqref{SDE2}. We will prove in the following that $(X,K)$ is pathwise unique, hence a strong solution. Recall \eqref{ODE} and \eqref{ODEalt}. Corollary \eqref{Cor:4.5} implies 
	$$0<\left(F^{-1}\right)'(x)\leq\max\limits_{x_0\leq x'\leq \bar{x}}\beta^{-1}\frac{D(F^{-1}(x'),x')}{N(F^{-1}(x'),x')-D(F^{-1}(x'),x')}<\infty,\quad\text{for all $x\in[x_0,\bar{x}]$},$$ because of the continuity of the functions $N$ and $D$, and the fact that $$N(F^{-1}(x),x)-D(F^{-1}(x),x) > 0,\quad \text{for any $x\in[x_0,\bar{x}]$},$$ which is due to Lemma \ref{lemma:ybar}, Proposition \ref{Prop:SolODE} and Lemma \ref{Implications}. Therefore, $\bar{F}^{-1}$ is (globally) Lipschitz continuous. 
	Now, fix $\omega\in\Omega$, and let $(\tilde{X},\tilde{K})$ and $(\hat{X},\hat{K})$ be two solutions of \eqref{SDE2}. The (global) Lipschitz continuity of $\bar{F}^{-1}$ and the second line of \eqref{SDE2} imply
	\begin{align}\label{EstimateK}
	\begin{split}
	\left|\tilde{K}_t-\hat{K}_t\right|&=\left|\sup\limits_{0\leq s \leq t}\big\{F^{-1}(\tilde{X}_s)-(\bar{y}-(y+\Delta))\big\}^+-\sup\limits_{0\leq s \leq t}\big\{F^{-1}(\hat{X}_s)-(\bar{y}-(y+\Delta))\big\}^+\right|\\
	&\leq\sup\limits_{0\leq s \leq t}\Big\{\left|F^{-1}(\tilde{X}_s)-F^{-1}(\hat{X}_s)\right|\Big\}\\
	&\leq \sup\limits_{0\leq s \leq t}\bar{K}\left|\tilde{X}_s-\hat{X}_s\right|\leq C_0\int_{0}^{t}\left|\tilde{X}_s-\hat{X}_s\right|+\left|\tilde{K}_s-\hat{K}_s\right|ds,
	\end{split}
	\end{align}
	for some constant $C_0>0$. Then, again with the second line of \eqref{SDE2} and \eqref{EstimateK}, we find for some constant $C_1>0$ the estimate
	\begin{align}\label{eq120}
	0\leq\left\Vert(\tilde{X}_t-\hat{X}_t,\tilde{K}_t-\hat{K}_t)\right\Vert\leq C_1\int_{0}^{t}\left|\tilde{X}_s-\hat{X}_s\right|+\left|\tilde{K}_s-\hat{K}_s\right|ds,
	\end{align}
	where $||\cdot||$ denotes the euclidean norm in $\mathbb{R}^2$. Now, Gr\"onwall's inequality yields
	\begin{align}\label{eq121}
	\left\Vert(\tilde{X}_t-\hat{X}_t,\tilde{K}_t-\hat{K}_t)\right\Vert\leq 0,
	\end{align} upon recalling that $t\mapsto X_t$ is continuous for any solution of \eqref{SDE2}. Thus, by \eqref{eq121}, pathwise uniqueness holds, and \eqref{SDE2} admits a unique strong solution.
	
	Finally, since $I^\star$ from \eqref{Istar} satisfies \eqref{eq2VT} and \eqref{eq3VT}, we conclude that $w$ identifies with $V$, and $I^\star$ is an optimal installation strategy by Theorem \ref{VerificationTheorem}.
\end{proof}
\section{Numerical Implementation}
\label{sec:NumImp}
\begin{figure}
	\subfigure[\label{Fig1a}The functions $F^{-1}$ and $M^{-1}$ with $\mu=0.2$.]{\includegraphics[width=0.49\textwidth]{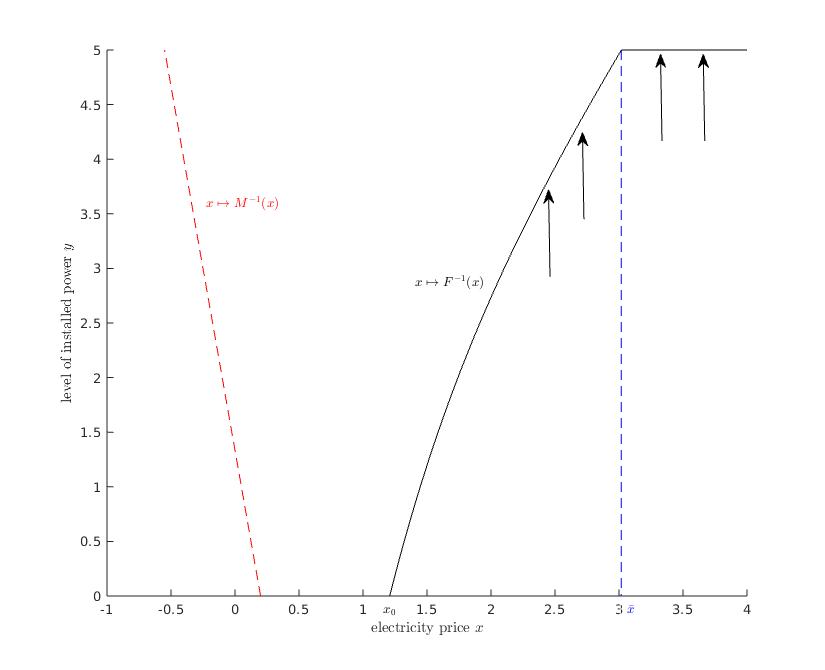}}
	\subfigure[\label{Fig1b}The functions $F^{-1}$ and $M^{-1}$ with $\mu=1.4$.]{\includegraphics[width=0.49\textwidth]{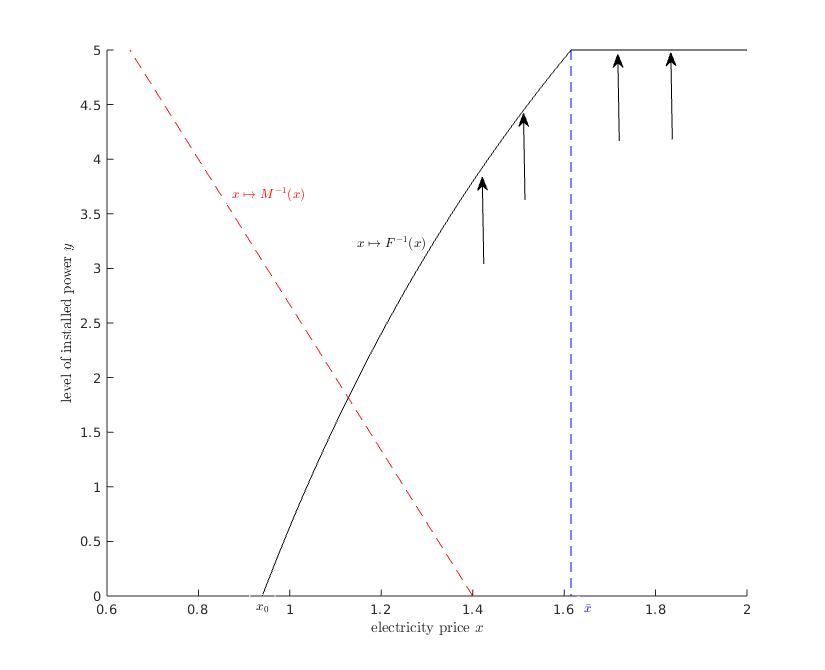}}
	\subfigure[\label{Fig1c}The functions $F^{-1}$ and $M^{-1}$ with $\mu=2.25$.]{\includegraphics[width=0.49\textwidth]{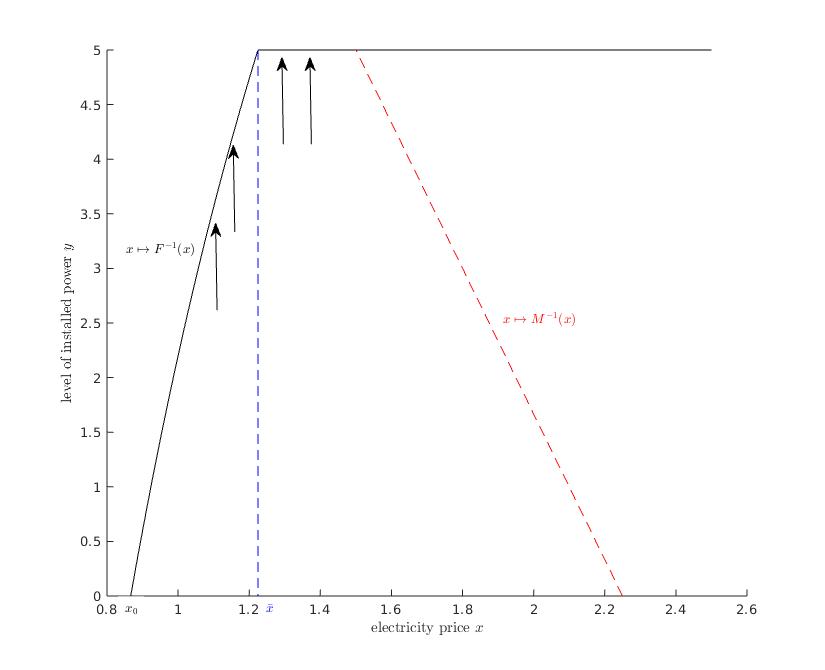}}
	\caption{Plots of the functions $x\mapsto F^{-1}(x)$ and $x\mapsto M^{-1}(x)$ with various values for $\mu$. The optimal installation strategy prescribes the following. In the region $\{(x,y) \in \mathbb{R} \times [0,\bar{y}):\, x < F(y)\}$ it is optimal not to install additional solar panels. Conversely if, at the initial time, $(x,y)$ is such that $x \geq F(y)$ and $y\in[0,\bar{y})$, then the (optimally controlled) process $(X,Y)$ should be pushed in direction $(0,1)$ as follows: for $x \geq \bar{x}$, the firm should immediately install the maximum number of panels, so to increase the level of installed power by $\bar{y}-y$ units. For $(x,y)$ such that $x\in [F(y),\bar{x})$, the firm should make an initial lump sum installation of size $F^{-1}(x)-y$, and then keep on making infinitesimal installations just preventing the price to exceed $F(y)$, until the maximum quantity of panels $\bar y$ is installed.}\label{Fig1} 
\end{figure}


The ordinary differential equation \eqref{ODE} cannot be solved analytically, but we are able to solve it numerically with {\sf MATLAB}. Figure \ref{Fig1} displays a plot of the inverse of the free boundary $F$ with three different values for the drift coefficient $\mu$. In particular we take those parameters' values as given in Table \ref{table2}, and $\mu\in\{0.2;1.4,2.25\}$.
\begin{table}[h]
	\centering
	\begin{tabular}   {| l | l | l | l | l | l |}
		\hline
		$\kappa$ & $\sigma$ & $\rho$ & $c$ & $\beta$ & $\bar{y}$ \\ \hline
		0.10 & 0.50 & 0.05 & 0.30 & 0.15 & 5  \\
		\hline
	\end{tabular}
	\caption{Parameters' values.}
	\label{table2}
\end{table}


The dashed sloped red line is a plot of the inverse of the function $M:[0,\infty)\mapsto(-\infty,\mu]$ given by $M(y):=\mu-\beta y$ (to which we shall refer as ``line of means''). The function $M$ provides the underlying mean-reversion level of the process $X^{x,y}$ depending on the level of installed power $y$. Figures \ref{Fig1a}, \ref{Fig1b} and \ref{Fig1c} show three different scenarios. The red line can {lie} entirely to the left or to the right of $F^{-1}$  (see Figure \ref{Fig1a} and Figure \ref{Fig1c}), or it can intersect $F^{-1}$ (see Figure \ref{Fig1b}). Notice that the position of the current mean reversion level in fact influences the expected time of the next action: if the red line is entirely to the left of $F^{-1}$ (i.e. the current mean reversion level is below $F(y)$ for any $y\in[0,\bar{y}]$), then the electricity price tends to move towards the line of means and therefore to stay below the firm's threshold, at which it starts to undertake the installation of additional solar panels. Conversely, the electricity price tends to move above the firm's threshold $F(y)$ for some $y\in[0,\bar{y}]$, if the red line intersects or lies in the installation region $\mathbb{I}$. Such a case in turn implies that the firm will increase its level of installed power faster.  In the limiting situation, that is when the red line is entirely on the right of $F^{-1}$, i.e. when the line of means lies entirely in $\mathbb{I}_2$, there is a
very high probability that $X$, if not already there and left uncontrolled, will enter into $\mathbb{I}_2$ in a very short time, so that in either case 
the firm would quickly install the maximum possible capacity $\bar{y}$.

The next proposition gives a characterization of when and how the line of means intersects the installation region $\mathbb{I}$ either at the free boundary or at its upper bound, i.e. at $\{(x,y)| y = \bar y, x \geq \bar x\}$. 

\begin{proposition}
Given the upper bound $\bar y$ for the solar panel cumulative level, and the corresponding free boundary $F(y)$ starting from $(\bar x,\bar y)$, the line of means $x = \mu - \beta y$:
\begin{enumerate}
\item has no intersection with the installation region $\mathbb{I}$ if $F(0) > \mu$;
\item intersects the boundary of $\mathbb{I}$ in the free boundary $F(y)$ if $F(0) \leq \mu$ and $\bar y \geq y^*$, where 
\begin{equation} \label{y^*}
y^* := \frac{1}{\beta (\rho + 2 \kappa)} \left( (\mu - \rho c)(\rho + \kappa) - \rho \frac{\psi(\mu)}{\psi'(\mu)} \right)
\end{equation}
\item intersects the boundary of $\mathbb{I}$ in its upper bound $y = \bar y$ if $\bar y \leq y^*$.
\end{enumerate}
\end{proposition}
\begin{proof}
For case (1), since the line of means $x = \mu - \beta y$ is decreasing in $y$ and the free boundary $F$ is increasing, there is no intersection if $\mu - \beta \times 0 = \mu < F(0)$. 

Let us now assume that $\mu \geq F(0)$ and discriminate between cases (2) and (3). The line of means $x = \mu - \beta y$ and the free boundary $x = F(y)$ have one or zero intersection according to whether $\bar x = F(\bar y)> \mu - \beta \bar y$ or not, respectively, i.e. whether 
$F(\bar y)+\beta \bar y =\tilde{x}(\bar{y})> \mu$, where we have written $\tilde{x}(\bar{y})$ in order to stress the dependency of $\tilde{x}$ on $\bar{y}$. Employing Lemma \ref{UniqueSol} and the implicit function theorem, we get 
$$\tilde{x}'(\bar{y})=\frac{\psi'(\tilde{x})\tilde{R}_y(\tilde{x},\bar{y})}{\psi''(\tilde{x})\left(c-\tilde{R}(\tilde{x},\bar{y})\right)}=-\frac{\beta(\rho+2\kappa)\psi'(\tilde{x})}{\rho(\rho+\kappa)\psi''(\tilde{x})\left(c-\tilde{R}(\tilde{x},\bar{y})\right)}>0,$$ 
where the strict inequality holds as $c-\tilde{R}(\tilde{x},\bar{y})<0$ by Lemma \ref{PropA} together with \eqref{Aytilde}.
Therefore, $F(\bar y)+\beta \bar y$ is increasing in $\bar y$, which means that there will exist a point $y^*$ such that there is an intersection for $\bar y > y^*$, and there is no intersection for $\bar y < y^*$. The point $y^*$ is characterized by the fact that the line of means $x = \mu - \beta y$ intersects both the free boundary $x = F(y)$ and the upper bound of the domain $y = \bar y$ in the same point $(\bar x, y^*)$. By Lemma \ref{UniqueSol} and its conclusion, the point $\bar x$ is characterized as $\bar x = \tilde x - \beta \bar y$, where $\tilde x$ is the solution of $H(\tilde x) = 0$, with $H$ defined in Equation \eqref{hxinf}. Thus, to find $y^*$ we must impose simultaneously
$$ \left\{ \begin{array}{l}
\bar x = \mu - \beta y^*, \\
\bar x = \tilde x - \beta y^*, \\
\psi'(\tilde x) (c - \tilde R(\tilde x, y^*)) + (\rho + \kappa)^{-1} \psi(\tilde x) = 0\\
\end{array} \right. $$
Thus, in this case $\tilde x = \mu$, and the third equation can be rewritten as
$$ \psi'(\mu) \left(c - \frac{\mu \kappa + \rho \mu - \beta (\rho + 2 \kappa) y^*}{\rho (\rho + \kappa)} \right) + (\rho + \kappa)^{-1} \psi(\mu) = 0 $$
This is a first-order algebraic equation for $y^*$, and the solution is easily obtained as in Equation \eqref{y^*}. 
\end{proof}

\begin{remark}
Checking where the upper bound $\bar y$ falls among the three cases above gives immediately a qualitative information on how the installation will proceed. In fact, in Case (1) there is always a very low probability for the state $(X_t,Y_t)$ to enter into the installation region $\mathbb{I}$, thus the installation will arrive to the upper bound $\bar y$ in a very long time. Conversely, in Case (3) there is a high probability that the initial state $(x,y)$ belongs to $\mathbb{I}_2$, and in this case at time 0 there will be a lump installation $\Delta$ as in Theorem \ref{Th:4.8} making that immediately $Y_0 = \bar y$. If this is not the case, there is still a very high probability for the state $(X_t,Y_t)$ to enter into the installation region $\mathbb{I}$, and any installation strategy will not be able to change this situation: in this case, as already said, the optimal strategy will keep on making infinitesimal installations just to prevent the price to exceed F(y), until the maximum quantity of panels $\bar y$ is installed. This, with a quite high probability, will happen in a short time. Case (2) lies obviously in between, and the fact that the initial installation $\Delta$ is sufficient to saturate the total capacity $\bar y$, or is even at all present, depends on the initial price level $X_0 = x$. 
\end{remark}

\begin{remark}
Unfortunately, to discriminate between Case (1) and the other two, one has to solve numerically Equation \eqref{ODE} in order to check whether $F(0) > \mu$ or not. Instead, discriminating between Cases (2) and (3) is much easier, as the point $y^*$ in Equation \eqref{y^*} is given explicitly in terms of initial parameters and known functions. 
\end{remark}

\subsection{Comparative Statics}\label{sec:CS}
\begin{figure}
	\subfigure[\label{Fig2a}The function $F^{-1}$ with $\sigma=0.5$ (black), $\sigma=0.6$ (red), $\sigma=0.7$ (blue), $\sigma=0.8$ (cyan).]{\includegraphics[width=0.48\textwidth]{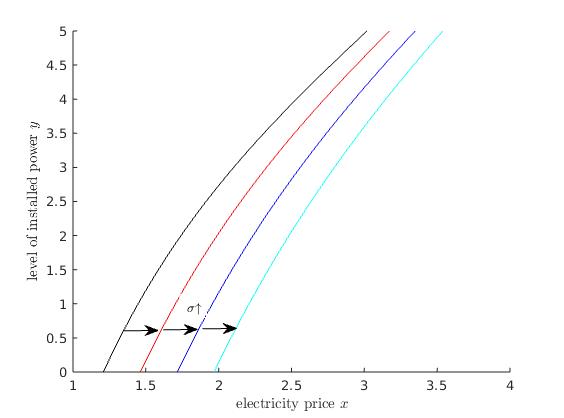}}\,\,
	\subfigure[\label{Fig2b}The function $F^{-1}$ with $\mu=0.2$ (black), $\mu=0.3$ (red), $\mu=0.4$ (blue), $\mu=0.5$ (cyan).]
	{\includegraphics[width=0.48\textwidth]{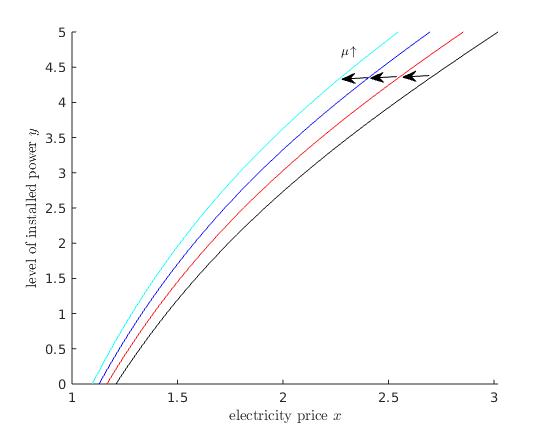}}
	\subfigure[\label{Fig2e}The function $F^{-1}$ with $\beta=0.15$ (black), $\beta=0.175$ (red), $\beta=0.2$ (blue), $\beta=0.225$ (cyan).]{\includegraphics[width=0.48\textwidth]{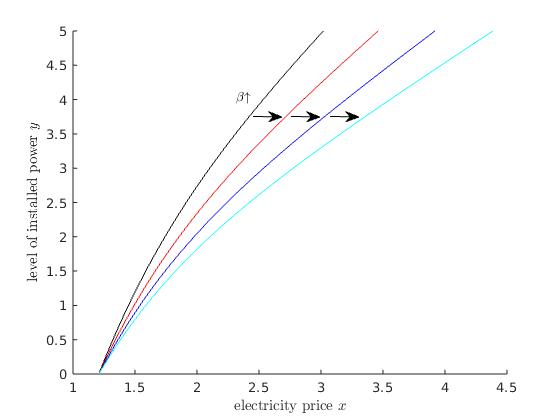}}
	\subfigure[\label{Fig2c}The function $F^{-1}$ with $\kappa=0.1$ (black), $\kappa=0.15$ (red), $\kappa=0.20$ (blue), $\kappa=0.25$ (cyan).]{\includegraphics[width=0.48\textwidth]{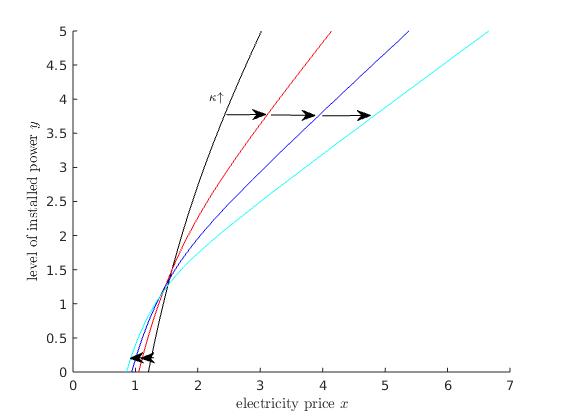}}
	\subfigure[\label{Fig2d}The function $F^{-1}$ with $c=0.3$ (black), $c=0.8$ (red), $c=1.3$ (blue), $c=1.8$ (cyan).]
	{\includegraphics[width=0.48\textwidth]{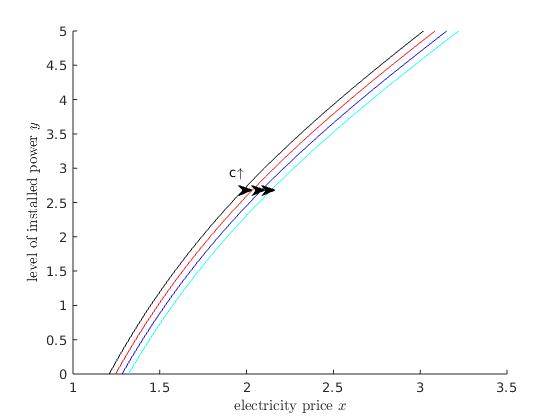}}
	\subfigure[\label{Fig2f}The function $F^{-1}$ with $\rho=0.035$ (black), $\rho=0.04$ (red), $\rho=0.045$ (blue), $\rho=0.05$ (cyan).]{\includegraphics[width=0.48\textwidth]{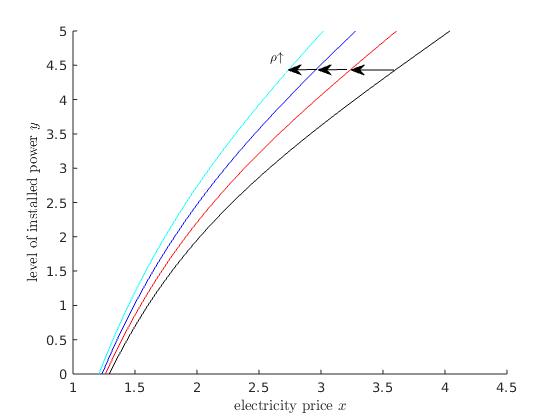}}
	\caption{Sensitivity of the function $x\mapsto F^{-1}(x)$ with respect to the model parameters. In each subfigure, the parameter values which are not varied are those provided in Table \ref{table}.\label{Fig2}}
\end{figure}
\begin{figure}
	\subfigure{\includegraphics[width=0.48\textwidth]{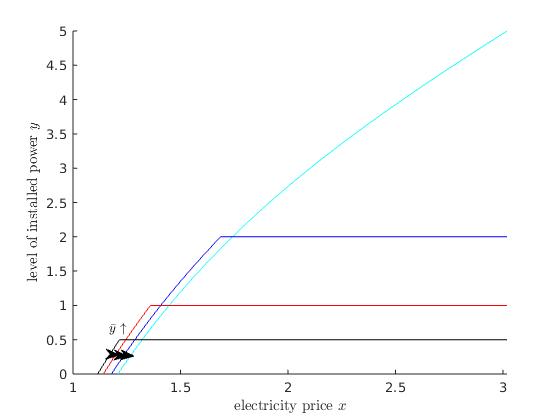}}
	\caption{Sensitivity of the function $x\mapsto F^{-1}(x)$ with respect to $\bar{y}$. In particular $\bar{y}=0.5$ (black), $\bar{y}=1$ (red), $\bar{y}=2$ (blue), $\bar{y}=5$ (cyan), and all the other parameter values are those provided in Table \ref{table}.\label{Fig3}}
\end{figure}
In this section, we study the sensitivity on the model parameters numerically. The preliminary parameters' values are given as in Table \ref{table}, and in the following we let each of those parameters vary within a particular set. The numerical results can be observed in Figure \ref{Fig2}.
\begin{table}[h]
	\centering
	\begin{tabular}   {|l|l|l|l|l|l|l|}
		\hline
		$\mu$ & $\sigma$ & $\kappa$ & $\rho$ & $c$ & $\beta$ & $\bar{y}$ \\ \hline
		0.20 & 0.50 & 0.1 & 0.05 & 0.30 & 0.15 & 5  \\
		\hline
	\end{tabular}
	\caption{Parameters' values for the numerical sensitivity analysis.}
	\label{table}
\end{table} 

We first study the behavior of the free boundary with respect to the volatility displayed in Figure \ref{Fig2a}. Here the volatility parameter $\sigma$ takes values in $\{0.5;0.6;0.7;0.8\}$, and we can observe that $F^{-1}$ is shifted to the right as $\sigma$ increases; that is the installation of additional panels is undertaken at higher prices. The firm might be afraid of receiving negative future prices due to higher uncertainty. This behavior is in line with the real options literature: when uncertainty increases, the agent is more reluctant to act, see for example \cite{mcdonald}.

Now, we let the mean-reversion level $\mu$ vary in $\{0.2;0.3;0.4;0.5\}$. Figure \ref{Fig2b} reveals that the critical threshold $F^{-1}$ moves to the left. A higher value for $\mu$ leads the firm to undertake the installation at lower prices. This observation can be explained by the fact that the company is eager to act earlier, the higher the expected future profits. 

In Figure \ref{Fig2e}, the impact parameter $\beta$ takes values in $\{0.15;0.175;0.2;0.225\}$, and as a consequence we find that $F^{-1}$ is shifted to the right as $\beta$ increases. We explain this observation by the fact that the impact of a higher electricity production on the future electricity prices is higher as $\beta$ increases. Therefore, the company starts to produce more electricity at higher prices, so to avoid lower (and possibly negative) electricity prices in the short run.

The dependency on $\kappa$ can be observed in Figure \ref{Fig2c}. Here, we let $\kappa$ taking values in $\{0.1;0.15;0.2;0.25\}$. We find that higher values for the mean reversion speed $\kappa$ leads the company to start installing solar panels at lower prices, but after some point, the company becomes more reluctant. This behavior can be explained by the fact that two effects play a role: on the one hand, a higher mean reversion speed reduces its ratio with respect to $\sigma$, the uncertainty is decreased, and hence a converse behavior with respect to Figure \ref{Fig2a} can be observed. On the other hand, a higher mean reversion speed also intensifies the impact of the company's actions on the price dynamics. Therefore, it behaves as in \ref{Fig2e}.

Figure \ref{Fig2d} shows the dependency on the proportional cost of installation $c$ which is valued in $\{0.3;0.8;1.3;1.8\}$. The shift is not parallel as one could suggest from the figure. The function $F^{-1}$ moves to the right, thus the company starts installing solar panels at higher prices. This observation is reasonable since the company waits for higher electricity prices to install additional solar panels whenever the proportional cost of installation increases.

Varying the discount factor $\rho$ in $\{0.035;0.04;0.045;0.05\}$, we find from Figure \ref{Fig2f} that $F^{-1}$ moves to the left, that is the company starts to install solar panels so to produce more electricity at lower prices. Clearly, a higher discount factor reduces the discounted future profits of the firm. Thus, the firm tends to produce more electricity earlier.

{Finally, we let $\bar{y}$ vary in $\{0.5;1;2;5\}$, and we observe that $F^{-1}$ moves to the right as $\bar{y}$ increases. Consequently, the possibility to increase the level of installed power up to a higher level makes the company more reluctant to act.}
\vspace{1cm}


\appendix
\section{Auxiliary Results}	
\label{App:A}

\renewcommand{\theequation}{A-\arabic{equation}}

\begin{lemma}\label{Properties}
	Let $\mathcal{L}$ denote the infinitesimal generator of the uncontrolled Ornstein-Uhlenbeck process \eqref{Xuncontrolled}, that is $\LL\equiv\LL^0$, where $\LL^y$, for $y\geq0$ be given and fixed, is the generator from \eqref{InfinitesOp}. Then the following hold true.
	\begin{itemize}
		\item [(1)] The strictly increasing{ positive} fundamental solution $\psi(\cdot)$, and the strictly decreasing{ positive} fundamental solution $\phi(\cdot)$ to the ordinary differential equation $(\LL-\rho) u=0$ are given by
		\begin{align}\label{eq34}
		\psi(x)&=e^{\frac{\kappa(x-\mu)^2}{2\sigma^2}}D_{-\frac{\rho}{\kappa}}\bigg(-\frac{x-\mu}{\sigma}\sqrt{2\kappa}\bigg),\\\label{eq38}
		\phi(x)&=e^{\frac{\kappa(x-\mu)^2}{2\sigma^2}}D_{-\frac{\rho}{\kappa}}\bigg(\frac{x-\mu}{\sigma}\sqrt{2\kappa}\bigg),
		\end{align}
		where 
		\begin{align}\label{eq35}
		D_\alpha(x):=\frac{e^{-\frac{x^2}{4}}}{\Gamma(-\alpha)}\int_{0}^{\infty}t^{-\alpha-1}e^{-\frac{t^2}{2}-xt}dt,\quad \alpha<0,
		\end{align} 
		is the cylinder function of order $\alpha$ and $\Gamma(\,\cdot\,)$ is the Euler's Gamma function.
		\item [(2)] Denoting by $\psi^{(k)}$ and $\phi^{(k)}$ the k-th derivative of $\psi$ and $\phi$, $k\in\mathbb{N}_0$, one has that  $\psi^{(k)}$ and $\phi^{(k)}$ are strictly convex and $\psi^{(k)}$ ($\phi^{(k)}$ respectively) identifies with the strictly increasing{ positive} (strictly decreasing{ positive} respectively) fundamental solution (up to a positive constant) to $(\LL-(\rho+k\kappa))u=0$. {In particular, it holds
			\begin{align}\label{Property1}
			\begin{split}
			\frac{\sigma^2}{2}\psi^{(k+2)}(x+{\beta}y)+\kappa\big((\mu-{\beta}y)-x\big)\psi^{(k+1)}(x+{\beta}y)-(\rho+k\kappa)\psi^{(k)}(x+{\beta}y)=0,
			\end{split}
			\end{align}
		for any $x\in\mathbb{R}$ and $y\geq0$}.
		\item [(3)] For any $k\in\mathbb{N}_0$, $\psi^{(k)}(x)\psi^{(k+2)}(x)-\psi^{(k+1)}(x)^2>0$ for all $x\in\mathbb{R}$.
		\item [(4)]\label{BasicAss}	For any $k\in\mathbb{N}_0$, the function $\Psi_k:\mathbb{R}\mapsto\mathbb{R}$ defined as $$\Psi_k(x)=\frac{\psi^{(k+1)}(x)^2}{\psi^{(k)}(x)\psi^{(k+2)}(x)},$$ is strictly increasing.
		\item [(5)]\label{} Denote by $\psi(\,\cdot\,;y)$ ($\phi(\,\cdot\,;y)$ respectively) the strictly increasing (strictly decreasing respectively) {positive} fundamental solution to $(\LL^y-\rho) u=0$ {for $y\geq 0$}. Then, {one can identify}
		$$ \psi(x;y) = \psi(x + \beta y), \qquad \phi(x;y) = \phi(x + \beta y).$$		
	\end{itemize}
\end{lemma}
\begin{proof}
	{The proofs of (1)-(3) can be found in \cite{Koch2}, cf. Lemma 4.3 therein, and $(4)$ has been proven to be valid in \cite{Koch}, cf. \emph{Step 1} in the proof of Theorem 3.1 therein. Moreover, (5) follows from (2), and in particular from equation \eqref{Property1} with $k=0$.}
\end{proof}

\begin{lemma}\label{Implications}
	For any $(y,z)\in\mathbb{R}\times\mathbb{R}$, we have the following implication
	\begin{align*}
	D(y,z)\geq 0\,\Rightarrow\,N(y,z) > D(y,z).
	\end{align*}
\end{lemma}
\begin{proof}
	Recall \eqref{DenOld}, and let $(y,z)\in\mathbb{R}^2$ be such that $D(y,z)\geq 0$. The previous inequality implies
	\begin{align}\label{DCon}
	(\rho+\kappa)\left(c-\tilde{R}(z,y)\right) \geq - \frac{{Q_0'(z)}}{Q_1(z)},
	\end{align}
	as $Q_1$ is strictly positive.
	
	In order to proceed with the proof, we introduce the function $\Phi:\mathbb{R}\times\mathbb{R}\mapsto\mathbb{R}$ such that
	\begin{align*}
	\Phi(z):=\psi''(z)Q_0(z)-\psi(z)Q_1(z).
	\end{align*}
	Employing Lemma \ref{BasicAss}-(4) with $k=0$, we find that $\Phi$ is strictly positive. Now, we use both \eqref{DCon} and the positivity of $\Phi$ to get 
	\begin{align}\label{NCond}
	\begin{split}
	N(y,z)-D(y,z)&=(\rho+\kappa)\left(c-\tilde{R}(z,y)\right)\Phi(z)+\frac{2(\rho+\kappa)}{\rho}\psi'(z)Q_0(z)-\psi(z) {  Q_0'(z)} \\
	&\geq { - \frac{Q_0'(z)}{Q_1(z)} \Phi(z)+\frac{2(\rho+\kappa)}{\rho}\psi'(z)Q_0(z)-\psi(z) Q_0'(z)} \\
	&= \left(\rho Q_1(z)\right)^{-1} Q_0(z) \Big[{ - \rho \psi''(z) Q_0'(z) + 2 (\rho + \kappa)  \psi'(z) Q_1(z)} \Big], 
	\end{split}
	\end{align}
	where we have rearranged terms after the inequality. To finish the proof, we exploit \eqref{Property1} with $k=0,1,2$ in \eqref{NCond}, so to obtain
	\begin{align}
	\begin{split}
	N(y,z)-D(y,z)\geq&\frac{\sigma^2}{2} \left(\rho Q_1(z)\right)^{-1}Q_0(z)\Big[\psi'''(z)Q_1(z)-\psi'(z)Q_2(z)\Big]>0,
	\end{split}
	\end{align}
	where the last inequality holds true upon recalling $Q_k>0$ and Lemma \ref{BasicAss}-(4) with $k=1$.
\end{proof}

\section{Proofs of Results from Section \ref{sec:PreResVerTheorem} and Section \ref{sec:Sol}}	
\label{App:B}

\renewcommand{\theequation}{B-\arabic{equation}}

\emph{Proof of Proposition \ref{GrowthV}.}
\vspace{0.25cm}\\
The proof employs arguments from the proof of Proposition 3.1 in \cite{Koch2} that are adjusted to our setting. In a first step we prove that \eqref{eq13} holds true, and then in a second step we show the monotonicity property of $V$.\vspace{0.25cm}

\emph{Step 1.} In order to prove the lower bound of $V$, we take the admissible (non-)installation strategy $I^0$, so to obtain for all $y\in[0,\bar{y}]$ 
\begin{align}\label{eq7}
V(x,y)\geq R(x,y)>
-K_1\big(1+|x|\big),
\end{align}
for some $K_1>0$.

To determine the upper bound of $V$, recall the uncontrolled price process $X^{x}$ from \eqref{Xuncontrolled}, and notice that by an application of It\^o's formula we find for any $\tilde{\rho}>0$
$$|e^{-\tilde{\rho} t}X^{x}_t|\leq |x|+\tilde{\rho}\int_0^te^{-\tilde{\rho} u}|X^{x}_u|du+\int_0^te^{-\tilde{\rho} u}\kappa(|\mu|+|X^{x}_u|)du+\bigg|\int_0^te^{-\tilde{\rho} u}\sigma dW_u\bigg|,$$
which in turn implies 
\begin{align}\label{eq40}
\mathbb{E}\bigg[\sup_{t\geq 0}e^{-\tilde{\rho} t}|X_t^{x}|\bigg]
\leq|x|+C_1\bigg(1+\int_{0}^{\infty}e^{-\tilde{\rho} u}\mathbb{E}\big[|X^{x}_u|\big]du\bigg)+\sigma\mathbb{E}\bigg[\sup_{t\geq 0}\bigg|\int_{0}^{t}e^{-\tilde{\rho} u}dW_u\bigg|\bigg],
\end{align}
for some $C_1>0$. An application of the Burkholder-Davis-Gundy inequality (cf. Theorem 3.28 in Chapter 3 of \cite{karatzas}) yields 
\begin{align}\label{eq41}
\mathbb{E}\bigg[\sup_{t\geq 0}e^{-\tilde{\rho} t}|X_t^{x}|\bigg]\leq |x|+C_1\bigg(1+\int_{0}^{\infty}e^{-\tilde{\rho} u}\mathbb{E}\big[|X^{x}_u|\big]du\bigg)+C_2\mathbb{E}\bigg[\bigg(\int_{0}^{\infty}e^{-2\tilde{\rho} u}du\bigg)^{\frac{1}{2}}\bigg].
\end{align}
for a constant $C_2>0$, and therefore
\begin{align}\label{supfinite}
\mathbb{E}\bigg[\sup_{t\geq 0}e^{-\tilde{\rho} t}|X_t^{x}|\bigg]\leq C\big(1+|x|\big),
\end{align}
for some constant $C>0$, since it follows from standard calculations that $\mathbb{E}\big[|X_u^{x}|\big]\leq C_3(1+|x|)$ for a constant $C_3>0$.

Now, for any $I\in\mathcal{I}^{\bar{y}}(y)$ we find by \eqref{supfinite}
\begin{align}\label{eq8}
\begin{split}
\mathcal{J}(x,y,I)&\leq\mathbb{E}\bigg[\int_{0}^{\infty}e^{-\rho t}X_t^{x,y,I}Y^{y,I}_t dt\bigg]\leq\mathbb{E}\bigg[\int_{0}^{\infty}e^{-\rho t}X_t^{x}Y^{y,I}_t dt\bigg]\\&\leq\mathbb{E}\bigg[\int_{0}^{\infty}e^{-\rho t}\big|X_t^{x}\big|Y^{y,I}_t dt\bigg]
\leq\bar{y}\mathbb{E}\bigg[\int_{0}^{\infty}e^{-\frac{\rho}{2} t}|e^{-\frac{\rho}{2} t}X_t^{x}|dt\bigg]\leq K_2\big(1+|x|\big),
\end{split}
\end{align}
for some $K_2>0$, and upon observing that $X^{x,y,I}\leq X^x$ $\mathbb{P}$-a.s. for any $I\in\mathcal{I}^{\bar{y}}(y)$. Finally, from \eqref{eq7} and \eqref{eq8}, we have that \eqref{eq13} holds with $K=\max(K_1,K_2)$.\vspace{0.25cm}

\emph{Step 2.} If $y=\bar{y}$, then the only admissible strategy is $I^0$, thus $V(x,\bar{y})=R(x,\bar{y})$. In order to prove that $x\mapsto V(x,y)$ is increasing, let $x_2>x_1$, and notice that one has $X^{x_2,y,I}_t\geq X^{x_1,y,I}_t$ $\mathbb{P}$-a.s.\ for any $t\geq 0$ and $I\in\mathcal{I}^{\bar{y}}(y)$. Thus $\mathcal{J}(x_2,y,I)\geq\mathcal{J}(x_1,y,I)$ which implies $V(x_2,y)\geq V(x_1,y)$.
\vspace{0.35cm}

\emph{Proof of Lemma \ref{PropA}.}
\vspace{0.25cm}\\
In the following, \emph{Step 1} proves the positivity and the monotonicity property of the function $A$, while \emph{Step 2} provides both the representation of $A$ and the lower bound of $F$.\vspace{0.25cm}

\emph{Step 1.} Recalling that $R_{yx}(x,y)=\left(\rho+\kappa\right)^{-1}$ for all $(x,y)\in\mathbb{R}\times[0,\bar{y}]$, we find from \eqref{cond5}
\begin{align}\label{A'y}
A'(y)=-\beta\frac{\psi''(\tilde{F}(y))}{\psi'(\tilde{F}(y))}A(y)-\frac{1}{(\rho+\kappa)\psi'(\tilde{F}(y))}=\mathcal{H}(\tilde{F}(y),A(y)),
\end{align}
where $\mathcal{H}:\mathbb{R}\times\mathbb{R}\mapsto\mathbb{R}$ is such that $$\mathcal{H}(\bar{F},A)=-\beta\frac{\psi''(\bar{F})}{\psi'(\bar{F})}A-\frac{1}{(\rho+\kappa)\psi'(\bar{F})}=-\frac{1}{(\rho+\kappa)\psi'(\bar{F})}\left(\beta(\rho+\kappa)\psi''(\bar{F})A+1\right).$$
In light of the boundary condition $w(x,\bar{y})=R(x,\bar{y})$, cf. Theorem \ref{VerificationTheorem}, we must have that
\begin{align}\label{cond8}
A(\bar{y})=0.
\end{align}
Due to \eqref{cond8} and the fact that $\mathcal{H}|_{\mathbb{R}\times[0,\infty)}$ is strictly negative as $\psi^{(k)}$ is strictly positive for any $k\in\mathbb{N}_0$, cf. Lemma \ref{Properties}-(2), we conclude that $A$ is both strictly positive and strictly decreasing.\vspace{0.25cm}

\emph{Step 2.}
Equations \eqref{cond4} and \eqref{cond5} lead to
\begin{align}\label{Ay}
A(y)=\beta^{-1}\times\frac{\psi'(F(y)+{\beta}y)\Big(c-R_y(F(y),y)\Big)+\left(\rho+\kappa\right)^{-1}\psi(F(y)+{\beta}y)}{\psi'(F(y)+{\beta}y)^2-\psi''(F(y)+{\beta}y)\psi(F(y)+{\beta}y)}.
\end{align}
Lemma \ref{Properties}-(3) ensures that the denominator of $A$ is nonzero. Now, the numerator on the right-hand side of \eqref{Ay} writes as
\begin{align*}
&\left(\beta\rho(\rho+\kappa)\right)^{-1}\left[(\rho+\kappa)\psi'(F(y)+{\beta}y)\left(c-R_y(F(y),y)\right)+\rho\psi(F(y)+{\beta}y)\right]\\
=&\left(\beta\rho(\rho+\kappa)\right)^{-1}\left[(\rho+\kappa)\Big(c\rho+\frac{\kappa\beta}{\rho+\kappa}y-F(y)\Big)\psi'(F(y)+{\beta}y)+\frac{\sigma^2}{2}\psi''(F(y)+{\beta}y)\right],
\end{align*}
upon using \eqref{Property1} with $k=0$. Hence, 
\begin{align}\label{Ayneueq}
A(y)=\left(\beta\rho(\rho+\kappa)\right)^{-1}\times\frac{(\rho+\kappa)\Big(c\rho+\frac{\kappa\beta}{\rho+\kappa}y-F(y)\Big)\psi'(F(y)+{\beta}y)+\frac{\sigma^2}{2}\psi''(F(y)+{\beta}y)}{\psi'(F(y)+{\beta}y)^2-\psi''(F(y)+{\beta}y)\psi(F(y)+{\beta}y)}.
\end{align} Since the denominator on the right-hand side of \eqref{Ayneueq} is strictly negative by Lemma \ref{Properties}-(3), we must have from the results of \emph{Step 1} that the numerator on the right-hand side of \eqref{Ayneueq} is also strictly negative: this is possible only if $$c\rho+\frac{\kappa\beta}{\rho+\kappa}y-F(y)<0,$$ as $\psi^{(k)}$ is strictly positive for any $k\in\mathbb{N}$. Hence, $F$ satisfies 
\begin{align}\label{condrev}
F(y)> c\rho+\frac{\kappa\beta}{\rho+\kappa}y\geq c\rho,\quad\text{for all }y\in[0,\bar{y}].
\end{align}
\vspace{0.35cm}

\noindent \emph{Proof of Proposition \ref{Prop:SolODE}.}\\
The proof is organised in two steps: in a first step, we provide a representation of the function $D$ that is used after. Then, in \emph{Step 2}, we show the existence of a strictly increasing maximal solution ${ \tilde F }$ of the ODE \eqref{ODE}, and prove (by a contradiction) that ${ \tilde F }$ in fact exists on the interval $[0,\bar{y}]$.\vspace{0.25cm} 

\emph{Step 1.} Recall \eqref{DenOld}, and let $\tilde{D}:\mathbb{R}\times\mathbb{R}\mapsto\mathbb{R}$ be a function which is given by 
\begin{align}\label{Dneu}
\tilde{D}(y,z)=\left[(\rho+\kappa)\psi(z)Q_0(z)\right]^{-1}D(y,z).
\end{align}
Then, where $\tilde{F}$ exists, we find upon employing \eqref{Aytilde} and \eqref{Aprimey}
\begin{align}\label{eq25}
\tilde{D}(y,\tilde{F}(y))=-{\beta}\psi'''(\tilde{F}(y))A(y)-\psi''(\tilde{F}(y))A'(y).
\end{align}
Now, from \eqref{Property1}, we derive
\begin{align}\label{eq21}
\psi^{(k+2)}(\tilde{F}(y))=-\frac{2\kappa}{\sigma^2}\left(\mu-\tilde{F}(y)\right)\psi^{(k+1)}(\tilde{F}(y))+\frac{2(\rho+k\kappa)}{\sigma^2}\psi^{(k)}(\tilde{F}(y)),\quad k\in\mathbb{N}_0.
\end{align}
Using \eqref{eq25} and the latter equation \eqref{eq21} with $k=0,1$, we obtain 
\begin{align}\label{eq23}
\begin{split}
\tilde{D}(y,\tilde{F}(y))
=&\frac{2}{\sigma^2}\Big[\kappa\left(\mu-\tilde{F}(y)\right)\big(\beta\psi''(\tilde{F}(y))A(y)+\psi'(\tilde{F}(y))A'(y)\big)\\
&-\rho\big(\beta\psi'(\tilde{F}(y))A(y)+\psi(\tilde{F}(y))A'(y)\big)-\kappa\beta\psi'(\tilde{F}(y))A(y)\Big]\\
=&\frac{2}{\sigma^2}\Big[\tilde{F}(y)-c\rho-\frac{(\rho+2\kappa)\beta}{\rho+\kappa}y-\kappa\beta\psi'(\tilde{F}(y))A(y)\Big],
\end{split}
\end{align}
where we have employed \eqref{cond4} and \eqref{cond5} for the last equality.\vspace{0.25cm}

\emph{Step 2.} Recall \eqref{ODE} and \eqref{ODEalt}. In the following, we denote by $\mathcal{D}_\mathcal{G}$ the domain of $\mathcal{G}$, that is $\mathcal{D}_\mathcal{G}=(\mathbb{R}\times\mathbb{R})\setminus\{(y,z)\in\mathbb{R}^2:D(y,z)=0\}$. Since $\psi^{(k)}$ is continuously differentiable for any $k\in\mathbb{N}$, the functions $N$ and $D$ are continuously differentiable respectively. Therefore, $\mathcal{G}(y,\cdot)$ is locally Lipschitz-continuous on its domain $\mathcal{D}_\mathcal{G}$ which is an open set. Hence, we find that the ODE \eqref{ODE} with boundary condition $\tilde{F}(\bar{y})=\tilde{x}$ admits a unique maximal solution $\tilde{F}$ on an interval $I_\text{max}=(y_-,y_+)$ with $\bar{y}\in I_\text{max}$. Since we want to show the existence and uniqueness of a solution on $[0,\bar{y}]$, it is enough to prove that $y_- < 0$. Following, for example, Theorem 2.10 in \cite{Barbu}, $y_-<\bar{y}$ is such that
\begin{itemize}
	\item [(i)] either $\lim\limits_{y\downarrow y_-}\left(||(y,\tilde{F}(y))||\right)^{-1}=0$,
	\item[(ii)] or $\lim\limits_{y\downarrow y_-}\inf\limits_{w\in\partial\mathcal{D}_\mathcal{G}}||(y,\tilde{F}(y))-w||=0$,
\end{itemize}
where $\partial\mathcal{D}_\mathcal{G}=\{(y,z)\in\mathbb{R}^2:D(y,z)=0\}$ is the boundary of the domain of $\mathcal{G}$, and $||\cdot||$ is a norm in $\mathbb{R}^2$.

Now, suppose $y_-\geq 0$. Notice that $N(y,\tilde{F}(y))>D(y,\tilde{F}(y))>0$ for all $y\in I_\text{max}$ by Lemma \ref{lemma:ybar} and Lemma \ref{Implications}, and therefore, we have $\tilde{F}'>\beta>0$ on $I_\text{max}$. Adjusting slightly the proof of Lemma \ref{PropA}, we find that $\tilde{F}$ is bounded from below on $(y_-,\bar{y}]$, and together with its monotonicity property, we must have that $\lim\limits_{y\downarrow y_-}\left(||(y,\tilde{F}(y))||\right)^{-1}>K$, for some $K>0$. Thus, in order to have a contradiction, it is left to prove that condition (ii) above is not satisfied, so to show $\lim\limits_{y\downarrow y_-}D(y,\tilde{F}(y))\neq 0$. Again, due to the boundedness of $\tilde{F}$ and the fact that both $Q_0$ and $\psi$ are strictly positive, we find $$\psi(\tilde{F}(y))Q_0(\tilde{F}(y))>K_1,\quad\text{for all }y\in(y_-,\bar{y}],$$ for some $K_1>0$. Therefore, upon recalling \eqref{Dneu}, we can complete the proof by showing that $\lim\limits_{y\downarrow y_-}\tilde{D}(y,\tilde{F}(y))\neq 0$. Lemma \ref{lemma:ybar} implies 
\begin{align}\label{tildeDybar}
\tilde{D}(\bar{y},\tilde{F}(\bar{y}))>0.
\end{align} Computing the total derivative of $\tilde{D}(y,\tilde{F}(y))$ with respect to $y\in I_\text{max}$, upon using \eqref{eq23}, gives
\begin{align}\label{H'y}
\begin{split}
\frac{d}{dy}\tilde{D}(y,\tilde{F}(y))&=\frac{2}{\sigma^2}\left[\tilde{F}'(y)\left(1-\kappa\beta\psi''(\tilde{F}(y))A(y)\right)-\frac{(\rho+2\kappa)\beta}{\rho+\kappa}-\kappa\beta\psi'(\tilde{F}(y))A'(y)\right]\\
&=\frac{2}{\sigma^2}\left(\tilde{F}'(y)-\beta\right)\left(1-\kappa\beta\psi''(\tilde{F}(y))A(y)\right),
\end{split}
\end{align}
where the last equality holds by an application of \eqref{cond5}. Next, we write the last coefficient in \eqref{H'y}, that is $1-\kappa\beta\psi''(\tilde{F}(y))A(y)$, as a function of ${G}:\mathbb{R}\times\mathbb{R}\mapsto\mathbb{R}$ defined as
\begin{align*}
&{G}(y,z)\\
&=\left((\rho+\kappa)Q_0(z)\right)^{-1}\left[(\rho+2\kappa)\psi(z)\psi''(z)-(\rho+\kappa)\psi'(z)^2+\kappa(\rho+\kappa)\left(c-\tilde{R}(z,y)\right)\psi'(z)\psi''(z)\right].
\end{align*}
Employing \eqref{Aytilde}, we get $1-\kappa\beta\psi''(\tilde{F}(y))A(y)={G}(y,\tilde{F}(y))$, and thus we have 
\begin{align}\label{TotDerDneu}
\frac{d}{dy}\tilde{D}(y,\tilde{F}(y))=\frac{2}{\sigma^2}\left(\tilde{F}'(y)-\beta\right)G(y,\tilde{F}(y)).
\end{align} Now, let $(y^\star,z^\star)\in\mathbb{R}\times\mathbb{R}$ be such that $\tilde{D}(y^\star,z^\star)=0$. We find from \eqref{Dneu} that $D(y^\star,z^\star)=0$. Hence, upon recalling \eqref{DenOld}, it holds
\begin{align}\label{eq24}
(\rho+\kappa)\left(c-\tilde{R}(z^\star,y^\star)\right)=-\frac{{ Q_0'}(z^\star)}{Q_1(z^\star)}.
\end{align}
Then, exploiting \eqref{eq24}, we obtain
\begin{align}\label{GProp}
\begin{split}
{G}({y}^\star,z^\star)
=&\bigg((\rho+\kappa)Q_0(z^\star)Q_1(z^\star)\bigg)^{-1}\times\bigg[(\rho+\kappa)\psi(z^\star)\psi'(z^\star)\psi''(z^\star)\psi'''(z^\star)\\
&-(\rho+2\kappa)\psi(z^\star)\psi''(z^\star)^3+(\rho+2\kappa)\psi'(z^\star)^2\psi''(z^\star)^2-(\rho+\kappa)\psi'(z^\star)^3\psi'''(z^\star)\bigg]\\
=&-\frac{\sigma^2}{2}\bigg((\rho+\kappa)Q_1(z^\star)\bigg)^{-1}Q_2(z^\star)
<0.
\end{split}
\end{align}
In \eqref{GProp} we have used: equation \eqref{Property1} with $k=0,1,2$ for the last equality, and the fact that $Q_1$ and $Q_2$ are strictly positive for the strict inequality.\\
Recalling that $\tilde{F}'-\beta>0$ on $I_\text{max}$, we conclude from \eqref{tildeDybar}, \eqref{TotDerDneu} and \eqref{GProp} that $\tilde{D}(y,\tilde{F}(y))$ cannot tend to zero as $y\downarrow y_-$. 
This completes the proof.




\medskip

\indent \textbf{Acknowledgments.} Financial support by the German Research Foundation (DFG) through the Collaborative Research Centre 1283 ``Taming uncertainty and profiting from randomness and low regularity in analysis, stochastics and their applications'' is gratefully acknowledged by the first author. We wish to thank Almendra Awerkin, Dirk Becherer, {Tiziano De Angelis,} Giorgio Ferrari, {Markus Fischer,} Wolfgang Runggaldier, {Thorsten Schmidt} and Mihail Zervos for useful comments.

\smallskip



\end{document}